\documentclass[a4paper,12pt,leqno]{article}
\makeatletter
\usepackage{amsmath}
\usepackage{amsfonts,delarray,amssymb,amsthm,a4,a4wide}
\usepackage[ansinew]{inputenc}
\newtheorem{theo}{Theorem}
\newtheorem{lem}{Lemma}[section]
\newtheorem{defi}{Definition}
\newtheorem{pro}{Proposition}[section]

\newtheorem{coro}{Corollary}[section]

\newtheorem{remark}{Remark}[section]
\def\Xint#1{\mathchoice
   {\XXint\displaystyle\textstyle{#1}}%
   {\XXint\textstyle\scriptstyle{#1}}%
   {\XXint\scriptstyle\scriptscriptstyle{#1}}%
   {\XXint\scriptscriptstyle\scriptscriptstyle{#1}}%
   \!\int}
\def\XXint#1#2#3{{\setbox0=\hbox{$#1{#2#3}{\int}$}
     \vcenter{\hbox{$#2#3$}}\kern-.5\wd0}}
\def\dashint{\Xint-}

\DeclareMathOperator{\supp}{Supp}\DeclareMathOperator{\diam}{diam}
\DeclareMathOperator{\dist}{dist}
\def\({\left(}
\def\){\right)}
\def\1{\mathbf{1}}

\def\anep{{\mathcal{C}_\ep^\alpha}}

\def\a{\alpha}

\def\Baep{{\mathcal{B}^\alpha_\ep}}
\def\Barep{{\mathcal{B}^{\alpha,r}_\ep}}
\def\Barhoep{{\mathcal{B}^{\alpha,\rho}_\ep}}
\def\Bbep{{\mathcal{B}^\beta_\ep}}
\def\BBrep{{\mathcal{B}^{B,r}_\ep}}
\def\Bdsep{{\mathcal{B}^{2\sqrt\ep}_\ep}}
\def\Bep{{\mathcal{B}_\ep}}
\def\Betaep{{\mathcal{B}^{\eta/2}_\ep}}

\def\Bkep{{\mathcal{B}^{r_k}_\ep}}
\def\Bkiep{{\mathcal{B}^{r_{k+1}}_\ep}}

\def\D{\displaystyle}
\def\Brepa{{\mathcal{B}^{r,\alpha}_\ep}}
\def\Brep{{\mathcal{B}^r_\ep}}
\def\Bsep{{\mathcal{B}^{\sqrt\ep}_\ep}}

\def\B{{\mathcal{B}}}

\def\cba{{\overline{C}_\alpha}}
\def\Cepa{{\mathcal{C}_\ep^\alpha}}
\def\Cepia{{\mathcal{C}_\ep^{i,\alpha}}}
\def\Cepi{{\mathcal{C}_\ep^i}}
\def\Cep{{\mathcal{C}_\ep}}

\def\curl{{\rm curl\,}}
\def\C{\mathcal{C}}

\def\div{\mathrm{div} \,}
\def\dt0{{{\frac{d}{dt}}_{|t=0}}}
\def\dtep{{d^t_\ep}}

\def\ep{\varepsilon}
\def\eaep{{e_\alpha}}
\def\fepB{{f_\ep^B}}
\def\fep{{f_\ep}}
\def\fa{{f_\a}}

\def\ga{{g_\alpha}}

\def\gepal{{\tilde g_\ep^\alpha}}
\def\gepp{{{(\tilde g_\ep^\alpha)}_+}}
\def\gepm{{{(\tilde g_\ep^\alpha)}_-}}
\def\gp{{{(g_\ep)}_+}}

\def\gepB{{g_\ep^B}}

\def\gep{{g_\ep}}

\def\GRC{{M_{R+C}}}
\def\GR{{M_R}}

\def\hal{\frac{1}{2}}

\def\indic{\mathbf{1}}
\def\indic{\mathbf{1}}

\def\lapr{{\Lambda^{\alpha,r}_\ep}}
\def\laprho{{\Lambda^{\alpha,\rho}_\ep}}
\def\Lap{{\Lambda^\alpha_\ep}}

\def\lep{{|\mathrm{log }\ \ep|}}

\def\lipom{{\text{\rm Lip}_\Omega}}

\def\loc{{\text{\rm loc}}}
\def\Lp{{L^p_\loc(\mr^2,\mr^2)}}

\def\l|{\left|}

\def\mc{\mathbb{C}}

\def\mr{\mathbb{R}}

\def\mz{\mathbb{Z}}
\def\nab{\nabla}
\def\naepp{{{n_\alpha}^2}}
\def\naep{{n_\alpha}}

\def\np{\nab^{\perp}}
\def\nuep{{\nu_\ep}}
\def\nua{{\nu_\a}}
\def\om{\Omega}

\def\P{\mathcal{P}}
\def\p{\partial}
\def\Q{\mathcal{Q}}

\def\ro{\rho}
\def\r|{\right|}

\def\sm{\setminus}

\def\Tae{{T_\ep^\alpha}}

\def\tgep{{\tilde g_\ep}}

\def\URC{{\mathbf U_{R+C}}}
\def\URc{{\mathbf U_{R-C}}}
\def\URI{{\mathbf U_{R+1}}}
\def\UR{{\mathbf U_R}}

\def\vp{\varphi}

\def\dl{\hal \lep}

\def\HE{\widehat{E}}

 \numberwithin{equation}{section}
\title{Improved Lower Bounds for Ginzburg-Landau
 Energies via Mass Displacement}
\author{Etienne Sandier and Sylvia Serfaty}

\begin{document}
\maketitle
\begin{abstract} We prove some improved  estimates for the
Ginzburg-Landau energy (with or without magnetic field) in two
dimensions, relating the asymptotic  energy of an arbitrary
configuration to its vortices and their degrees, with possibly
unbounded numbers of vortices. The method is based on a
localisation of  the ``ball construction method" combined with
 a mass displacement idea which allows to compensate for negative
 errors in the ball construction estimates by energy ``displaced"
 from close by.
  Under good conditions,
 our main estimate allows to get a lower bound on the energy
 which includes  a finite order ``renormalized energy"  of vortex interaction,
 up to the best possible precision i.e.
   with only a  $o(1)$ error per vortex, and is complemented by local compactness results on the vortices.
    This is used crucially in our forthcoming paper
    \cite{papier2}. It can also serve
   to provide
 lower bounds for weighted Ginzburg-Landau energies.
\end{abstract}

\noindent
{\bf keywords:} Ginzburg-Landau, vortices, vortex balls construction, renormalized energy, mass displacement.

\section*{Introduction}
 We  are interested in proving lower bounds and compactness results for Ginzburg-Landau type energies of the form
$$ G_\ep(u,A)= \hal\int_{\om_\ep} |\nab_A u|^2 + (\curl A)^2 +
\frac{(1-|u|^2)^2}{2\ep^2}$$
where   $\ep $ is a small parameter, $u$ is a
complex-valued function called ``order parameter", $A$ is $\mr^2 $
valued and is the vector potential of the magnetic field $h:=\curl
A$ and  $\nab_A = \nab- iA$. Here the domain of integration
$\om_\ep$ is a smooth bounded domain in $\mr^2$, which  can depend
on $\ep$. We are interested in particular in the case where
$\om_\ep$  gets large as $\ep \to 0$. Note that one may set $A
\equiv 0 $   to recover the simpler Ginzburg-Landau energy
$$E_\ep(u) = \frac{1}{2}\int_{\om_\ep} |\nab u|^2  + \frac{(1-|u|^2)^2}{2\ep^2}$$
without magnetic field. Our results apply to this energy functional
by making this trivial choice of $A$.

The Ginzburg-Landau energy is a famous model for superconductivity.
  In this model  the order-parameter $u$ often has  quantized
  vortices,
   which are the zeroes of $u$ with nonzero topological degree.
     Obtaining ansatz-free lower bounds for $G_\ep$
     in terms of  the vortices of $u$ has proven to be crucial
      in studying the asymptotics of minimizers of $G_\ep$,
       in particular via $\Gamma$-convergence methods.

The first study establishing  lower bounds  for Ginzburg-Landau was
the work of Bethuel-Brezis-H\'elein \cite{bbh} for solutions to the
Ginzburg-Landau equations without magnetic field  with energy
$E_\ep$ bounded by $C \lep$. Such an energy bound ensures that the
total number of vortices remains bounded as $\ep \to 0$. This was
later improved and extended in two different directions by Han-Shafrir \cite{hs} and  Almeida-Bethuel \cite{ab} for arbitrary configurations, still  with a number of vortices that remains bounded. The main limitation of such estimates is that the error terms  blow up as the number of vortices gets large.  
Then,
 Jerrard \cite{jerr} and Sandier \cite{sandier} introduced the
 ``ball construction method", which provides lower bounds
 in terms of vortices for arbitrary configurations, allowing
  unbounded numbers of vortices and much larger energies.
 This is crucial for many applications,
   since energy minimizers of the functional with applied magnetic field
    do not always satisfy a $ C \lep$ bound on their energy.
Subsequent refinements of the ball construction method were given
(see for example \cite{livre} Chap. 4 for a recent result).
The lower bound provided by the ball construction method also provides a crucial compactness result on the vorticity (roughly the sum of Dirac masses at the vortex centers, weighted by their degrees), these are the so-called ``Jacobian estimates", see Jerrard-Soner \cite{js} and \cite{livre} Chap. 6 and references therein. They say roughly that the vorticity is controlled by $\frac{1}{\lep} $ times the energy.
 For other subsequent  works  refining  those results in a slightly
different direction, see \medskip   also \cite{ss4,jspirn,st}.

In a way our objective here can be seen as obtaining next order terms  (order 1 as opposed to order $\lep$)  in such  estimates, both energy estimates and compactness results.

For a given $(u, A)$, let us  define the energy density
$$ e_\ep(u,A)=\hal\( |\nab_{A} u|^2 + (\curl A)^2 +
\frac{(1-|u|^2)^2}{2\ep^2}\).$$ If $(u,A)$ is clear from the context
and defined on  a set $E$, we will often use the abbreviation
$e_\ep(E)$ for $\int_E e_\ep(u,A)$, and $e_\ep$ for the density
$e_\ep(u,A)$.
We then introduce the measure
$$f_\ep:= e_\ep - \pi\lep \sum_B d_B \delta_{a_B}$$ where the $a_B$'s are the centers of the vortex balls constructed via Jerrard's and Sandier's ball construction,  the $d_B$'s are the degrees of the balls and $\delta$ is the Dirac mass.
Calculating $\int f_\ep$ corresponds to subtracting off the cost of all vortices  from the total energy: what remains should  then correspond to the interaction energy  between the vortices, which we can call ``renormalized energy" by analogy with \cite{bbh}.
In order to obtain next order estimates of the energy $G_\ep$, we show here lower bounds on the energy $\int f_\ep$, as well as coerciveness properties of $f_\ep$, which say, roughly, that $f_\ep $, or in other words, the renormalized energy, suffices to control the  vorticity. (This is again to be compared with the previous ball construction and Jacobian estimate, where the vorticity is controlled  \medskip by $e_\ep/\lep$).

The motivation for this is our joint paper
\cite{papier2} where we establish a ``next order"
$\Gamma$-convergence result for the Ginzburg-Landau energy with
applied magnetic field, and derive a limiting interaction energy
between points in the plane, thus making the link to the question of  the
famous Abrikosov lattice (the Abrikosov lattice is a hexagonal lattice of vortices in superconductors observed in experiments and predicted by Abrikosov).
More precisely, we show in \cite{papier2} an asymptotic expansion for the minimal energy of the form
$$\min G_\ep= G_\ep^N + N \min W+o(N)$$
where $N\gg 1$ is the optimal number of vortices (determined by the intensity of the applied field), $G_\ep^N$ is a constant of order $N^2$ (the leading order estimate) and $W$ is a renormalized energy governing the pattern formed by the vortices after blow-up at the scale $\sqrt{N}$.
Moreover, we show  that the patterns formed by the vortices of minimizers after this blow-up  minimize $W$ (almost surely, in some sense). We prove in addition that among lattice configurations (of fixed volume), $W$ is uniquely minimized by the hexagonal lattice. The natural conjecture is that this lattice  is also a minimizer among all point configurations, and if this were proved, it would completely justify the emergence of the Abrikosov hexagonal lattice.

 To achieve this, with an error only $o(N)$, we needed
 lower bounds
on the cost of vortices with a  precision  $o(1)$
per vortex (with  still a possibly infinite number of vortices), which is finer than
 was available in the literature.  We also needed to control the (local number of) vortices by the renormalized energy. In fact the energy density we end up having to analyze in \cite{papier2} is exactly  $f_\ep$, and we need to be able to control the vortices through it.

  The other problem we need to overcome for \cite{papier2} is that $f_\ep$ is  obviously not positive or even bounded below, and this prevents from applying standard lower semi-continuity ideas, and the abstract scheme for $\Gamma$-convergence of 2-scale energies which we introduce in \cite{papier2}. This reflects the fact that the energy $e_\ep$ is not exactly where the vortices are, as we will explain below.
  The remedy which we implement here, is that we can ``deform" $f_\ep$ into an energy density $g_\ep$ which is bounded below and enjoys nice coerciveness properties.
To accomplish this we show that we  can transport  the positive mass in $f_\ep$ into the support of the negative mass in $f_\ep$,  with mass travelling at most at fixed finite distances   (say distance 1), and so that the result of the operation,  $g_\ep$, is bounded  below.
This is done by using the following rather    elementary transport lemma:
   \begin{lem}\label{transport0} Assume $f$ is a finite  Radon measure on a compact set $A$,  that $\om$ is open  and   that for any positive  Lipschitz function  $\xi$ in $\lipom(A)$, i.e. vanishing on $\om\sm A$,
$$\int \xi \,df\ge - C_0 |\nab\xi|_{L^\infty(A)}.$$
Then there exists a Radon measure  $g$ on $A$ such that  $0 \le g \le f_+ $  and such that
$$\|f-g\|_{\lipom(A)^*}\le C_0.$$
\end{lem}
Thus what is needed is a control on the negative part of $f_\ep$, which will be provided by  the ball construction lower bounds and  additional improvements of it.

The norm $\|f_\ep - g_\ep\|_{\lipom(\om)^*}$ will measure how far  mass has been displaced in the process. This control appears in Theorem \ref{th1} below and more particularly Corollary \ref{coro1}.
Since $\int g_\ep$ will be close to $\int f_\ep$, it also can be seen as a renormalized energy.
  Since $g_\ep$ is bounded below, we can then hope that it enjoys nice coerciveness properties, we can in fact
   obtain the desired  compactness results which allow to control  the vorticity  locally by $g_\ep$. This will be the object of Theorem  \ref{th2} below.

Finally, let us point out that our results can in principle serve to obtain  lower bounds for weighted Ginzburg-Landau \medskip energies, see Remark \ref{rem1.4}.

Let us now describe a little bit the  method that we use,  which will allow to control the negative part of $f_\ep$ as needed.

 The best  vortex ball construction lower bound  on $e_\ep$ available (such as that  in \cite{livre} Chap. 4) is of the following type: given $(u_\ep, A_\ep)$ and any (small) number $r$,  there exists a family of disjoint closed balls $\mathcal{B}$ covering all the zeros of $u_\ep$, the sum of the radii of the balls being bounded above by $r$, and such that
\begin{equation}\label{bornancienne}
\int_{\cup_{B \in \mathcal{B} }   B} e_\ep(u_\ep,A_\ep) \ge \pi D\( \log
\frac{r}{ \ep D }- C\),
\end{equation}
 where $D=  \sum_{B\in \mathcal{B}    } |d_B|$ with $d_B = \deg(u_\ep, \p B) $ if $B\subset \om$ and $0$ otherwise.
We shall reprove here in Proposition \ref{jerr} a   version of this result using Jerrard's ball construction.

This above estimate says that
  a vortex of degree $d$ costs an energy  at least $\simeq \pi |d|\lep$, but  this is only really true when the vortex is well isolated from other vortices and from the boundary, and if there are not too many of them locally, as the factor $r/D$ in the logarithm above somewhat reflects:  an ideal lower bound would be
$$ e_\ep(B) \ge \pi |d_B| \( \log \frac{r}{\ep}- C\),$$
 and compared to this, the lower bound above contains a negative error $- \pi D\log D$  which tends to $- \infty$ if the total number of vortices becomes large when $\ep \to 0$. In truth, this ideal lower bound cannot hold in general  as  can be seen in the case of $n$ vortices of degree $1$ all  positioned regularly near the boundary of the domain, a case where \eqref{bornancienne} is optimal.

    Moreover the energy density $e_\ep$ is not localized exactly where the vortices are: vortices can be viewed as points, while their energy is spread over annular regions around these points.
    The ball construction lower bounds such as  \eqref{bornancienne}   capture well the energy which lies very near the vortices, but some energy is missing from it, in particular when vortices accumulate locally around a point.
     The missing energy in that case can be recovered by
   the method of ``lower bounds on annuli"  which we introduced in \cite{ss0} and re-used  in \cite{livre}, Chap. 9. It is based on  the following: let $B(x_0, r_1)\backslash B(x_0, r_0)$ be an annulus which contains no
   zeros of $u$, roughly speaking we have
 $$e_\ep\(B(x_0, r_1) \backslash B(x_0, r_0) \)   \ge \pi
 D^2 \log \frac{r_1}{r_0}$$ where $D= \deg(u, \p B(x_0, r_1)) = \deg(u, \p B(x_0, r_0))$.
 In other words, if a fixed size ball
  in the domain contains some large degree $D$ of vorticity, then there is an energy  of order $D^2$ lying not in that ball, but in a  thick enough annulus around that ball. This energy  of order $D^2$  should suffice to ``neutralize" the error term  $-\pi D \log D$ found above through the ball construction.
   However, it lies at a certain (finite) distance from the center of the vortices.
      The main technique  is then to combine
in a systematic way the ball construction lower bounds and the ``lower bounds on annuli", in order to recover enough energy.

   Let us finally emphasize a technical difficulty. Since we want a local control on the vortices,
 the lower bound \eqref{bornancienne} provided by the ball construction is not quite sufficient because it  cannot be localized in general, i.e. we cannot deduce a bound for $\int_B e_\ep$ for each $B\in \mathcal{B}$. It is only possible when a matching upper bound on the  total in (\ref{bornancienne}) is known,
see Proposition \ref{jerr} for more details.

 The idea to remedy this difficulty is
 to ``localize" the construction, i.e. split the domain into pieces on which one expects to have a bounded vorticity, then apply the ball construction on each piece, and paste together the constructions and lower bounds obtained this way, whose error terms will now be bounded below by a constant. However, this is not completely easy:
one needs to localize the construction and still get a global covering of the vortices by balls while preserving the disjointness of the balls.
  In applications, trying to split the domain
 into pieces where  the vorticity is expected to be bounded
  leads us to splitting the domain into very small (as $\ep \to 0$)
   pieces. Equivalently  after rescaling one can consider
    very large domains cut into bounded size pieces.
    In other words, in order
     to be able to treat the case where the vortex density becomes large, we need to be able to treat the case of unbounded domains as $\ep \to 0$.

 This is precisely what we do in this paper: we consider possibly large
 domains.  This way we may in practice rescale our domains as
   much as needed  until the local density of vortices remains bounded as $\ep \to 0$.  We consider vortex ball constructions obtained over coverings of $\om_\ep$ by domains of fixed size, and we work at pasting together these lower bounds while combining them  with
   the method of lower bounds on annuli, as explained above, and finally retrieving ``finite numbers of vortices" estimates (of \cite{bbh} type) which bound from below  the energy  $f_\ep$ or $g_\ep$ by  the exact renormalized energy of \cite{bbh} type up to only $o(1)$ errors.

\section{Statement of the main results}

In this paper we will deal with families
 ${(u_\ep,A_\ep)}_\ep$ defined on domains $\{\om_\ep\}_\ep$
 in $\mr^2$ which become large as $\ep\to 0$. The example we have in mind is $\om_\ep = \lambda_\ep \om$ where $\om$ is a fixed bounded smooth domain and $\lambda_\ep\to +\infty$ as $\ep\to 0$, but we  don't need to make any particular  hypothesis on $\{\om_\ep\}_\ep$, which could even be a fixed bounded domain.

Next we introduce some notation.

For $E\subset\mr^2$ we let
$$\widehat{E}= \{x\in \om_\ep, \dist(x, E) \le 1\}.$$
We also define, for any real-valued or vector-valued  function $f$
in $\om_\ep$,
$$\widehat f(x) = \sup\{|f(y)|, y\in B(x,1)\cap\om_\ep\}.$$
Note that both $\widehat f$ and $\HE$ depend on $\ep$, but the value of $\ep$ will be clear from the context. The choice of $1$ in the definitions is arbitrary but constrains the choice of other constants below.

In all the paper, $f_+ $ and $f_- $ will denote the positive and
negative parts of a function or measure, both being positive
functions or measures, and $\|f\|$ is the total variation of $f$.
If $f$ and $g$ are two measures then
 $f \le g$ means  that $g-f$ is a positive measure.

Given a family $\{(u_\ep,A_\ep)\}_\ep$, where $u_\ep:\om_\ep\to\mc$ and $A_\ep:\om_\ep\to\mr^2$ we define the {\em currents} and {\em vorticities} to be
$$j_\ep= (iu_\ep, \nab_{A_\ep} u_\ep),\quad \mu_\ep = \curl j_\ep + h_\ep,$$
where  $(a,b)= \hal (a\bar{b}+ \bar{a}b)$ and $h_\ep = \curl A_\ep$ is the {\em induced magnetic
field .}

We denote by $\lipom(A)$ the set of Lipschitz functions on $A$ which are $0$ on $\om\sm A$, and  let
$\|f\|_{\lipom(A)^*} = \sup \int
\xi\,df$, the supremum being taken over functions $\xi\in\lipom(A)$ such that $|\nab \xi|_{L^\infty(A)} \le 1$.

We say a family $\{f_\a\}_{\a}$ is subordinate to a cover $\{A_\a\}_{\a}$ if $Supp(f_\a) \subset A_\a$ for every $\a.$

Despite the slightly confusing notation, the covering $A_\a$ will have nothing to do with the magnetic gauge $A_\ep$. Also, the densities $f_\a $ and $g_\a$, as well as $n_\a$ and $\nu_\a$ will implicitly depend on $\ep$, and should be really $f_{\ep, \a}$ and $g_{\ep, \a}$, etc,  but for simplicity we do not indicate this dependence.

\begin{theo} \label{th1} Let $\{\om_\ep\}_{\ep>0}$ be a family
of bounded open sets in $\mr^2$. Assume that    $\{(u_\ep, A_\ep)\}_\ep$, where $(u_\ep,A_\ep)$ is defined over $\om_\ep$,  satisfies  for some  $0<\beta<1$ small enough
\begin{equation}\label{bornenrj}
G_\ep(u_\ep,A_\ep)  \le \ep^{-\beta}.\end{equation}
Then the following holds, for $\ep$ small enough:
\begin{enumerate}
\item (Vortices) There exists a measure $\nu_\ep$, depending only on
$u_\ep$ (and not on $A_\ep$) of the form $2\pi \sum_i d_i
\delta_{a_i}$ for some points $a_i \in \om_\ep$ and some integers
$d_i$ such that, $C$ denoting a generic constant independent of
$\ep$,
\begin{equation}\label{estjac}
\|\mu_\ep - \nu_\ep\|_{(C^{0,1}_0(\om_\ep))^*} \le
 C \sqrt\ep G_\ep(u_\ep, A_\ep),\end{equation}
 and for any measurable set $E$
$$|\nu_\ep|(E)\le C  \frac{e_\ep(\HE)}\lep .$$

\item (Covering) There exists a cover $\{A_\a\}_\a$ of $\om_\ep$ by open sets with
 diameter and  overlap number bounded by a universal constant, and measures
  $\{\fa\}_\a$, $\{\nua\}_\a$ subordinate to this cover such that,
  letting  $f_\ep := e_\ep - \dl\nu_\ep$ ,
$$f_\ep  \ge \sum_\a \fa ,\quad  \nuep  = \sum_\a  \nua, \qquad \nu_{\a_1} \perp \nu_{\a_2} \ \text{for } \ \a_1\neq \a_2.$$

\item (Energy transport) Letting  $\naep := \|\nua\|/2\pi$,
for  each $\a$  the following holds:
If $\dist(A_\a,{\om_\ep}^c)>\ep $ there exists a measure  $\ga\ge - C $  such that   either
\begin{equation}\label{cas1}
\|\fa - \ga\|_{\lipom(A_\a)^*}\le C \naep\( 1 + \beta\lep\) \quad \text{and}\quad \ga(A_\a)\ge c \naep\lep,
\end{equation}
or
\begin{equation}\label{cas2}
\|\fa - \ga\|_{\lipom(A_\a)^*} \le C \naep\(1+\log\naep\) \quad \text{and}\quad \ga(A_\a)\ge c \naep^2 - C \naep,
\end{equation}
where  and $c, C>0 $ are positive universal constants.

If $\dist(A_\a,{\om_\ep}^c)\le \ep $ there exists $\ga\ge 0$  such that for any function $\xi$
\begin{equation}\label{casbord}
\int  \xi \,d(\fa - \ga) \le  C \naep\(|\nab \xi|_{L^\infty(A_\a)} + \beta  \lep|\xi|_{L^\infty(A_\a)}\).
\end{equation}

\item (Properties of $g_\ep$) Letting $g_\ep =  f_\ep + \sum_\a(\ga -\fa) $ it holds that
\begin{equation}\label{minmajgh}-C \le   g_\ep \le e_\ep + \dl (\nu_\ep)_-, \end{equation}
and for any measurable set $E\subset \om_\ep$,
\begin{equation}\label{gmoins}(g_\ep)_-(E)  \le  C\frac{e_\ep(\HE)}{\lep},\quad \gp(E)\le  C e_\ep(\HE).\end{equation}
Moreover, assuming   $|u_\ep|\le 1$ in $\om_\ep$ and that $E+B(0,C)\subset\om_\ep$, for some $C>0$ large enough, then  for every $p<2$,
\begin{equation}\label{prosurj} \int_{E} |j_\ep|^p\le C_p \(\gp(E+B(0,C)) + |E|\).\end{equation}
\end{enumerate}
\end{theo}

The third item admits, or rather implies the following form, from which the covering $\{A_\a\}_\a$ is hidden.

\begin{coro}\label{coro1} Under the hypothesis above and using the same notation, for every $0<\eta\le 1$ we have if $\ep>0$ is small enough: First, for every Lipschitz function $\xi$ vanishing on $\p \om_\ep$
\begin{equation}\label{err}
\int_{\om_\ep} \xi \,d(g_\ep - f_\ep) \le  C\int_{\om_\ep}
\widehat{\nab \xi}\, \left[ d|\nu_\ep| +    (\beta+\eta) \,d\gp  +
\frac{|\log\eta|^2}{\eta}\,dx \right] + C\beta  \int_{\widehat{ \p
\om_\ep }} \widehat{\xi}    e_\ep.
\end{equation}
Second, if $d(E,\p\om_\ep)>C$ then
\begin{equation}\label{controlenu}
|\nu_\ep|(E)\le C\(\eta \gp(\HE) + \frac1\eta  |\HE| + \frac{e_\ep(\HE\cap\widehat{ \p \om_\ep })}{\lep} \).
\end{equation}
\end{coro}
The point in introducting the extra parameter $\eta$ is that we want to be able to use only a small $\eta$-fraction  of the ``remaining" energy $g_\ep$ to control the error $f_\ep - g_\ep$ between the original energy and the displaced one.
This corollary is obtained by  simply summing the relations \eqref{casbord} and controlling $n_\a $ and $n_\a \log n_\a$  by a small fraction of ${n_\a}^2$ through the
elementary relations
$$x\log x \le \eta x^2 +C \frac{\log^2 \eta}{\eta}\qquad  2x\le\eta x^2+\frac{1}{\eta}$$ and then  controlling  ${n_\a}^2$ by
$\ga(A_\a)$ via \eqref{cas1} or \eqref{cas2}.

\begin{remark} \label{rree1} If we let $\eta = 1$ and choose $E$ to be at distance at least
$1$ from   $\p\om$  then   (\ref{err}) and
(\ref{controlenu}) reduce to \begin{equation} \label{fgx} \int
\xi\, d(f_\ep - g_\ep) \le C\int_{\om_\ep} \widehat{\nab \xi}\,
\left[ d\gp + \,d|\nu_\ep|\right] \end{equation} and
$$|\nu_\ep|(E) \le C \((g_\ep )_+(\HE) + |\HE|\).$$ If one takes $\xi= \chi_R $ to be a positive cut-off function
supported in $B(0,R)$ and $\equiv 1$ in $B(0,R-1)$ then  the right-hand side in \eqref{fgx} scales like a boundary term i.e.  like $R$ as $R$ gets large, while the left-hand side scales like an interior term.
\end{remark}


\begin{remark} \label{rem3.2} Assume we have proved the above Theorem and Corollary. Then, given $\{(u_\ep,A_\ep)\}_\ep$ and $\{\om_\ep\}_\ep$  satisfying the hypothesis, we may consider for some fixed $\sigma>0$ the rescaled quantities $\tilde\ep = \ep/ \sigma$, $\tilde x = x/\sigma$ and let
$$ \tilde u_\ep(\tilde x) = u_\ep(x),\quad \widetilde{A}_\ep(\tilde x) =  \sigma A_\ep(x),\quad  \widetilde{\om}_\ep = \om_\ep/\sigma.$$
Then, letting $h = \curl A$ and $\tilde h = \curl\widetilde A$, we have
$$ e_\ep^\sigma (u,A) := \sigma^2 \(\hal|\nab_A u|^2 + \frac{\sigma^2}2 h^2+ \frac1{4\ep^2}(1-|u|^2)^2\) = \hal|\nab_{\widetilde A}\tilde u|^2 + \frac12 \tilde h^2+ \frac1{4\tilde \ep^2}(1-|\tilde u|^2)^2.$$
We may then apply the Theorem to the tilded quantities, yielding a measure $\tilde g_\ep$. Then if we let $g_\ep(x) = \tilde g_\ep(\tilde x)$, the measure $g_\ep$ will satisfy the properties stated in Theorem~\ref{th1} and Corollary~\ref{coro1}, with $e_\ep$ replaced by $e_\ep^\sigma$ (and with a different $C$)  provided we modify the definition of $\HE$ to
$$\HE = \{x\mid \dist(\tilde x,\widetilde{E})<1\} = \{x\mid \dist(x,E)<\sigma\},$$
(note that we can keep the original definition provided $\sigma\le 1$).

Then  we may add to both $e_\ep$ and $g_\ep$ the quantity $\(\hal - \frac{\sigma^2}2\){h_\ep}^2$ and obtain in this manner a new $g_\ep$  satisfying the listed properties and --- for the particular choice $\sigma^2 = \frac12$ --- the lower bound
\begin{equation}\label{rem51} g_\ep\ge \frac{{h_\ep}^2}4 - C.\end{equation}
We will then usually assume when applying Theorem~\ref{th1} that  this lower bound holds as well as the other conclusions of the theorem.
\end{remark}

The next result shows how $g_\ep$ has the desired coerciveness properties, and behaves like the renormalized energy. Indeed,  under the assumption that the family $\{g_\ep\}_\ep$ is bounded on compact sets (recall that the domains become increasingly large as $\ep\to 0$) we have  compactness results for the vorticities and currents, and lower bounds on $\int g_\ep$  (hence $\int f_\ep$ via \eqref{err}) in terms of the renormalized energy $W$.

Before stating that result, we introduce some additional notation.
We denote by $\{\UR\}_{R>0}$ a  family of sets in
$\mr^2$  such that   for some constant $C>0$ independent of $R$
\begin{equation}\label{hypsets}
\text{  $\UR+ B(0,1)\subset \URC$ and
$\URI\subset \UR+B(0,C)$}.\end{equation} For example $\{\UR\}_{R>0}$ can be the family $\{B_R\}_{R>0}$ of balls centered at $0$ of radius $R$.

Then  we  use the notation $\chi_{\UR}$ for  cutoff functions satisfying, for some $C$ independent of $R$,
\begin{equation}
\label{defchi}
|\nab \chi_{\UR}|\le C \quad \supp(\chi_{\UR}) \subset \UR \quad \chi_{\UR}(x)=1 \ \text{if } \dist(x, \UR^c) \ge 1.\end{equation}

Finally, given a vector field $j:\mr^2\to\mr^2$ such that $\curl j =
2\pi\sum_{p\in\Lambda} \delta_p + h$ with $\Lambda $, where $h$ is
in $L^2_\loc$ and $\Lambda $ a discrete set, we define the {\em
renormalized energy} of $j$ by
$$W(j)= \limsup_{R \to \infty}
\frac{W(j, \chi_{B_R})}{|B_R|}, $$
where for any $\chi$
\begin{equation}\label{WR}W(j, \chi) =
\liminf_{\eta\to 0} \( \hal\int_{\mr^2 \sm \cup_{p\in\Lambda}
B(p,\eta) }\chi |j|^2 +  \pi \log \eta \sum_{p\in\Lambda} \chi (p)
\). \end{equation}
Various results on $W$, in particular on its
minimizers,  are proved in \cite{papier2}. Note in particular that if we assume $\div j = 0$, then the $\liminf$ in \eqref{WR} is in fact a limit, because in this case $j = \np H$ with $\Delta H = 2\pi\delta_p + h$ in a neighbourhood of $p$, and thus $H = \log|\cdot - p| + f$ with  $f\in H^1$ in this neighbourhood.

\begin{theo}\label{th2} Under the hypothesis of Theorem~\ref{th1}, and assuming  $|u_\ep|\le 1$ in $\om_\ep$ we have the following.

\begin{enumerate}
\item Assume  that  $\dist(0 , \p \om_\ep ) \to + \infty$ as $\ep\to 0$ and that, for any $R>0$,
\begin{equation}\label{hypcoer}
\limsup_{\ep\to 0} g_\ep(\UR)\,dx<+\infty,\end{equation}
where $\{\UR\}_R$ satisfies (\ref{hypsets}).

Then, up to extraction of  a subsequence, the  vorticities $\{\mu_\ep\}_\ep$ converge in $W^{-1,p}_\loc(\mr^2)$ to a measure $\nu$ of the form $2\pi\sum_{p\in\Lambda}\delta_p$, where
$\Lambda$ is a discrete subset of $\mr^2$,
  the  currents $\{j_\ep\}_\ep$
converge weakly in $\Lp$ for any $p<2$ to $j$, and the induced fields $\{h_\ep\}_\ep$ converge weakly in $L^2_{\loc}(\mr^2)$ to $h$ which are such that
$$\curl j = \nu - h    \quad \text{in} \ \mr^2.$$

\item  If we replace the assumption \eqref{hypcoer} by  the stronger assumption
\begin{equation}\label{hypcoerplus}
\limsup_{\ep \to 0} g_\ep(\UR) <C R^2,\end{equation}
where $C$ is independent of $R$, then the limit $j$ of the currents satisfies, for any $p<2$,
\begin{equation} \label{lpunif} \limsup_{R\to +\infty} \dashint_{\UR}|j|^p\,dx<+\infty.\end{equation}
Moreover for every family $\chi_{\UR}$ satisfying (\ref{defchi})  we have
\begin{equation}\label{limlim}
 \liminf_{\ep\to 0} \int_{\mr^2}  \frac{\chi_{\UR} }{|\UR|} \,d\gep  \ge \(
\frac{W(j,\chi_{\UR})}{|\UR|} + \hal \dashint_{ \UR} h^2  +\frac{ \gamma
}{2\pi} \dashint_{\UR} h \) + o_R(1) ,
\end{equation}
where $\gamma$ is a constant defined below and $o_R(1)$ is function tending to $0$ as $R\to +\infty$.
\end{enumerate}
\end{theo}

\begin{remark} The constant $\gamma$ in \eqref{limlim} was introduced in \cite{bbh} and  may be  defined by
$$\gamma=\lim_{R \to \infty}\( \hal  \int_{B_R} |\nab u_0|^2
+ \frac{(1-|u_0|^2)^2}{2} -\pi \log R\) ,$$ where $u_0 (r,\theta) =
f(r) e^{i \theta} $  is the unique (up to translation and rotation)
radially symmetric degree-one vortex (see \cite{bbh,miro}).
\end{remark}

\begin{remark} \label{rem1.4} Lower bounds immediately follow from this theorem. Indeed $f_\ep$ is  the energy density minus the energetic
cost of a vortex, and $f_\ep-g_\ep$ is controlled by Theorem \ref{th1}, see also Remark 
\ref{rree1}.
  This, combined with the lower bound  (\ref{limlim}) shows that in
 good cases  the averages  
over large  balls of
$ f_\ep$  are bounded below by $W$ plus explicit constants, which proves a sharp lower
bound for the energy with a $o(1)$ order error, \`a la
Bethuel-Brezis-H\'elein \cite{bbh}.

The bound \eqref{err} may also be interpreted as a
lower bound for the Ginzburg-Landau energy with weight.
 Assuming a fixed domain $\om$ and  $G_\ep(u_\ep,A_\ep)<C\lep$
  for instance, and that $\mu_\ep \to 2\pi \sum_{i=1}^n
   \delta_{a_i}$, where $a_i\in\om$, then by blowing up by a
   factor independent of $\ep$  we may assume the points are at
    distance $2$, say, from the boundary and then if $\xi$ is a
    fixed positive  weight we may multiply it by a cutoff $0\le \chi\le 1$
     equal to zero on $\widehat{ \p \om}$ and  equal to $1$ at each $a_i$. Then \eqref{err} becomes
$$\int_\om \xi e_\ep\ge \pi \lep \sum_{i=1}^n \xi(a_i) + \int\chi
\xi\,dg_\ep - C\int \widehat{\nab (\chi\xi)}\, \left[ d|\nu_\ep| +
(\beta+\eta) \,d\gp  +  \frac{|\log\eta|^2}{\eta}\,dx \right]. $$
Typically, there will be an upper bound for the energy which
 implies that $\gp(\om) < C$   and since also $g_\ep \ge - C$, the integrals on the
 right-hand side may be bounded below by a constant independent of
  $\ep$.
\end{remark}

The paper is organized as follows: In Section \ref{sec2} we state
without proof the result on lower bounds via Jerrard's ball
construction (the proof is postponed to Section~5) which we adapt
for our purposes, and explain how we use it on a covering of
$\om_\ep$ by a collection $U_\a$ of balls of finite size.  In
Section 3, we present the tool used to transport the negative part
of $f_\ep$ to absorb it into the positive part, and deduce Theorem \ref{th1}. In
Section 4, we prove  Theorem \ref{th2}. Finally in
Section 5, we prove the ball-construction lower bound.
\\

{\bf Ackowledgements :}  Etienne Sandier was supported by the Institut Universitaire de France, Sylvia Serfaty by an NSF CAREER award and a EURYI award.

\section{Use of the ball construction and coverings of the
domain}\label{sec2}

The first step consists in performing a ball construction in
$\om_\ep$ in order to obtain lower bounds. This follows essentially
the method of Jerrard  \cite{jerr},  the difficulty being that we
are not allowed more than an error  of order one per vortex. This is
hopeless if the total number of vortices diverges when $\ep\to 0$,
hence we need to localize the construction in pieces of $\om_\ep$
small enough for the number of vortices in each piece to remain
bounded as $\ep\to 0$.

\subsection{The ball construction lower bound}

We start by stating the result of  Jerrard's ball construction in a
version  adapted to our situation, in particular including the
magnetic field. The proof is postponed to Section 5. In all what
follows, if $\mathcal{B}$ is a collection of balls, $r(\mathcal{B})$
denotes the sum of the radii of the balls in the collection. In all
the sequel we will sometimes abuse notation by writing $\B$ for
$\cup_{B\in\B} B$, i.e. identify the collection of balls and the set
it covers.

\begin{pro} \label{jerr} There exists $\ep_0, C>0$ such that if  $U\subset\mr^2$, $\ep\in (0,\ep_0)$, and $(u_\ep,A_\ep)$ defined on $U$  is such  that
$$G_\ep(u_\ep, A_\ep) \le \ep^{-\beta}, $$
where $\beta\in (0,1)$, the following holds.

For every   $r \in (C \ep^{1-\beta},\hal)$,  there exists
 a collection of disjoint closed balls $\mathcal{B}$  depending only
  on $u_\ep$ (and not
 on $A_\ep$) such that, letting $U_\ep = \{x\mid d(x,U^c)>\ep\}$,
\begin{enumerate}
\item $\left\{ x\in U_\ep\mid   |u_\ep(x)|< \hal \right\} \subset \B$.
\item $r(\mathcal{B}) \le r$.
\item For any  $2\le\overline{C}\le (r/\ep)^\hal $ it holds that either
$$e_\ep(\mathcal{B}  \cap U) \ge \overline{C} \log \frac{r}{ \ep} $$
 or
$$\text{$\forall B \in \mathcal{B}$ such that $B\subset U_\ep$},\quad e_\ep(B) \ge \pi  |d_B|\( \log \frac{r}{\ep \overline{C}}  -C\),$$
where $d_B=\deg(u_\ep, \p B) $.
\end{enumerate}
\end{pro}

A  natural choice of $\overline{C}$ above is $\pi D$, where $ D=
\sum_{B\in\mathcal{B}} |d_B|$ and we have  let $d_B = 0$ if $B\not\subset U_\ep$. With this choice we find in all cases
$$e_\ep(\mathcal{B}  \cap U)\ge  \pi D    \( \log \frac{r}{\ep D }
- C\)$$ i.e. we recover the same lower bound as in \cite{livre},
Theorem 4.1, mentioned in the introduction as (\ref{bornancienne}).
The reason why we don't  simply use that theorem directly is that we
need to keep the dichotomy above, and thus a lower bound localized
in each ball.

\subsection{Localizing the ball construction}
For any $\ep>0$ we construct an open cover $\{U_\alpha\}_\alpha$ of $\om_\ep$ as follows: We consider the collection $\B$ of balls of radius $\ell_0$ --- where $\ell_0\in(0,\frac18)$ is to be chosen below, small enough but  independent of $\ep$ --- centered at the points of $\ell_0\mz^2$. The cover consists of the open sets $\om_\ep\cap B$, for $B\in\B$.

This cover depends on $\ep$, but the maximal number of
 {\em neighbours} of a given $\alpha$ --- defined as the indices $\beta$ such that $U_\alpha\cap U_\beta\neq\varnothing$ --- is bounded independently of $\ep$ by an integer we denote by $m$ (in fact $m=9$). Note that $m$ also bounds the {\em overlap number} of the cover, i.e. the maximal number of $U_\alpha$'s to which a given $x$ can belong. There is also $\ell>0$ independent of $\ep$ which is a Lebesgue number of the cover, i.e. such that for every $x\in\om_\ep$, there exists $\alpha$ such that $B(x,\ell)\cap\om_\ep\subset U_\alpha$ or, equivalently, $\dist(x,\om_\ep\cap U_\alpha^c)\ge \ell$.

Assuming $\beta<1/4$, and  applying Proposition~\ref{jerr} to
$(u_\ep,A_\ep)$ in $U_\alpha$ for every $\alpha$ we obtain,
since $\sqrt\ep>C \ep^{1-\beta}$ if $\ep$ is small enough,
a collection $\Barep$ for every $\sqrt\ep\le r\le 1/2$.

%

If $\rho$ is chosen small enough depending on $\ell$ and $m$ only,
 thus less than a universal constant,  we may extract from
  $\cup_\alpha\Barhoep$ a subcollection $\Bep$ such that any
  two balls $B$, $B'$ in $\Bep$ satisfy $\om_\ep\cap B\cap B' = \varnothing$. We will say $\Bep$ is {\em disjoint in $\om_\ep$}:

\begin{pro} Assume  $\rho\le\ell/(8m)$. Then, writing in short
 $\Baep$ instead of $\Barhoep$, there exists a subcollection of  $\cup_\alpha \Baep$ --- call it $\Bep$ --- which is disjoint in $\om_\ep$ and such that
\begin{equation}\label{SE}\{|u_\ep|\le 1/2\}\cap\{x\mid \dist(x,{\om_\ep}^c)>\ep\}\subset \cup_{B\in\Bep} B.\end{equation}
Moreover,  for every $B\in \Bep\cap\Baep$ we have  $B\cap\om_\ep = B\cap U_\alpha$ and
$$\dist(B,{\om_\ep}^c)>\ep\iff \dist(B,{U_\alpha}^c)>\ep.$$
\end{pro}
\begin{proof}

Assume   $C = \om_\ep \cap \(B_1\cup\dots\cup B_k\)$ is
 a connected component of $\om_\ep\cap\(\cup_\alpha \Baep\)$.
 Reordering if
necessary, we may assume that
$B_i\cap(B_1\cup\dots\cup B_{i-1})\neq\varnothing$ for every
 $1\le i\le k$.  There exists $x\in \om_\ep\cap B_1$ and
 $\alpha$ such that $\dist(x,\om_\ep\cap U_\alpha^c)\ge \ell$.
 Then $\dist(B_1,\om_\ep\cap U_\alpha^c)> 3\ell/4$. Assume
$$\dist(B_1\cup\dots\cup B_{i-1},\om_\ep\cap U_\alpha^c)\ge
 \frac{3\ell}4.$$
Then $\dist(B_i,\om_\ep\cap U_\alpha^c)> \ell/2$ hence for
every  $1\le j\le i$ the ball $B_j$  belongs to $\Bbep$, where $\beta$ is a neighbour of $\alpha$. It follows that $r_1+\dots+r_i\le m\rho\le \ell/8$, where $r_i$ is the radius of $B_i$, and we deduce that $B_1\cup\dots\cup B_i\subset B(x,\ell/4)$ and then
$$\dist(B_1\cup\dots\cup B_{i},\om_\ep\cap U_\alpha^c)\ge \frac{3\ell}4.$$
We have thus proved by induction that $C\subset U_\alpha$  and even that $\dist\( C, \om_\ep\cap U_\alpha^c\)\ge 3\ell/4$ for every $i$.

We delete from $\{B_1,\dots, B_k\}$ the balls which do not belong to
$\Baep$ and call $C'$ the union of the remaining balls.  If $y$ belongs to
$$C\cap \{|u_\ep|\le 1/2\}\cap \{x\mid \dist(x,{\om_\ep}^c)>\ep\}$$
then,  since $\dist\(C, \om_\ep\cap U_\alpha^c\)\ge 3\ell/4$
and $\dist(y,{\om_\ep}^c)>\ep$, provided $\ep< 3\ell/4$
 we have that $\dist(y,{U_\alpha}^c)>\ep$ hence
  $y$  belongs to some ball $B\in \Baep$ (since $\Baep$ covers the set
  $\{|u_\ep|\le \hal \} \cap \{\dist (x, U_\alpha^c) >\ep\}$),
   thus $y\in C'$.  The balls in $C'$ are disjoint in $\om_\ep$ since they belong to the collection $\Baep$ which is itself disjoint in $\om_\ep$.

Performing this operation on each connected component of $\om_\ep\cap\(\cup_\alpha \Baep\)$ we thus obtain a collection $\Bep$ which covers $ \{|u_\ep|\le 1/2\}\cap \{x\mid \dist(x,{\om_\ep}^c)>\ep\}$ and is disjoint in $\om_\ep$. Moreover, if $B\in  \Bep\cap\Baep$ then $\dist\( B, \om_\ep\cap U_\alpha^c\)\ge 3\ell/4$ hence  $B\cap\om_\ep = B\cap U_\alpha$ and
$$ \dist(B,{\om_\ep}^c)>\ep
\iff  \dist(B,{U_\alpha}^c)>\ep.$$
\end{proof}

The value $\rho$ will be fixed to some value smaller than
$\ell/8m$ and independent of $\ep$, to be specified below. The above proposition  provides
us for any $\ep>0$ small enough with collections of balls $\Bep$ and
$\Baep$. We will also need the following
\begin{defi}\label{defi1} For any $\sqrt\ep \le r\le \rho$, and any $B\in\Baep$, we let $\BBrep$ be the collection of balls in $\Barep$ which are included in $B$. Then we let
$$\Brep = \cup_{B\in \Bep} \BBrep.$$
It is  disjoint in $\om_\ep$ and covers  the set
 $\{|u_\ep|\le 1/2\}\cap \{x\mid \dist(x,{\om_\ep}^c)>\ep\}$
  and of course if $B\in  \Brep\cap\Barep$, then
  $B\cap\om_\ep = B\cap U_\alpha$ and $$
   \dist(B,{\om_\ep}^c)>\ep \iff   \dist(B,{U_\alpha}^c)>\ep.$$
\end{defi}

In other words, the disjoint collection $\Bep$ permits us to
construct disjoint collections of smaller radius by discarding from
$\Barep$ those balls which are inside a ball discarded from
$\Barhoep$. The collection $\Bsep$ should be seen as the collection
of ``small balls" and $\Bep$ (obtained from $\Barhoep$) as the
collection of ``large balls". We will sometimes also use the
collection of the intermediate size balls $\Brep$  with $\sqrt{\ep}
\le r \le \rho$.

Finally we let
\begin{equation}\label{nu}
\nuep =\sum_{\substack{B\in\Bsep\\ \dist(B,{\om_\ep}^c)>\ep}} 2\pi d_B \delta_{a_B},\quad |\nuep|
= \sum_{\substack{B\in\Bsep\\ \dist(B,{\om_\ep}^c)>\ep}}2\pi |d_B| \delta_{a_B},
\end{equation}
where $a_B$ is the center of $B$, and $d_B$ denotes the winding
number of $u_\ep/|u_\ep|$ restricted to $\partial B$.  This is the
$\nu_\ep$ given by the  conclusion of the theorem. Note that since
the balls only depend on $u_\ep$ (and not on $A_\ep$), $\nu_\ep$
satisfies the same.
 If $B$ is any ball which does not cross the boundary of balls in $\Bsep$ and $\dist(B, {\om_\ep}^c)>\ep$ then
$\nuep(B) = 2\pi d_B$. From the Jacobian estimate (see \cite{js} or
the version in \cite{livre}, Theorem 6.1) we have that
(\ref{estjac}) is satisfied. We also have (recall that $\|f\|$
denotes the total variation of a measure)

\begin{lem}\label{lembornepetits} There exists $\ep_0>0$ such that if $\beta<1/4$ in \eqref{bornenrj} and  $\ep<\ep_0$ then
$$ |\nuep|(E)\le 16 \ \frac{ e_\ep(\om_\ep\cap\HE)}\lep$$
for any measurable set $E$, so that choosing  $E = \om_\ep$ and taking  logarithms,
\begin{equation}\label{bornepetits} \log\|\nu_\ep\|\le \beta\lep + C. \end{equation}
\end{lem}

\begin{proof} We use the properties of $\B_\ep^{\alpha,\sqrt\ep}$. Letting $\overline C = \(\sqrt\ep/\ep\)^\hal = \ep^{-\frac14}$, it is impossible when $\ep$ is small enough that $e_\ep(\om_\ep\cap \B_\ep^{\alpha,\sqrt\ep})\ge \overline C \log(\sqrt \ep/\ep)$ since we assumed that $e_\ep(\om_\ep)\le \ep^{-\beta}$. Thus  Proposition~\ref{jerr} implies that, for every $B\in \B_\ep^{\alpha,\sqrt\ep}$ such that $\dist(B, {U_\alpha}^c)>\ep$,
$$e_\ep(B)\ge \pi|d_B|\(\log \ep^{-\frac14} - C\)\ge \frac\pi 8|d_B|\lep,$$
if $\ep$ is small enough. If, moreover, $B\in\Bsep$, then from
Definition~\ref{defi1} we have
$$\dist(B, {U_\alpha}^c)>\ep \iff  \dist(B, {\om_\ep}^c)>\ep$$
Hence for any set $E$, using \eqref{nu} and the fact that balls in $\Bsep$ have radius smaller than $1/2$ if $\ep$ is small enough,
$$|\nuep|(E) \le \sum_{\substack{B\in\Bsep\\\dist(B,{\om_\ep}^c)>\ep \\B\cap E\neq\varnothing}}|\nuep|(B)\le 16 \frac{e_\ep(\om_\ep\cap\HE)}{\lep}.$$
\end{proof}

\begin{defi}
For any $\alpha$ we let $\nua$ denote the restriction of $\nuep$ to the balls in $\Bep\cap \Baep$ and
$\naep =\frac{\|\nua\|}{2\pi}$, so that
$$\nuep = \sum_\a \nua,\quad  \naep = \sum_{B\in\Bep\cap\Baep} \frac{|\nuep|(B)}{2\pi}, \quad \|\nuep\| =
2\pi \sum_\a\naep.$$
We also define
\begin{equation} \label{cba}
 \cba =\begin{cases} \max \(M\naep,\D \frac{3\eaep}\lep\)  & \text{if $\naep\neq 0$},\\ 2 & \text{otherwise,}\end{cases}
\end{equation}
where $M$ is a large universal constant to be chosen later and
$$ \eaep =  \sum_{B\in \Baep} e_\ep(B\cap U_\a).$$
\end{defi}
Note that $\naep$ is the sum of the absolute values of the degrees of the small balls included in
the large balls of $\Baep$.

We  have the following
\begin{pro} \label{goodbad} There exists $\ep_0,C_0>0$ such that if $\beta<1/4$ in \eqref{bornenrj} and  $\ep<\ep_0$, $\ep^\hal<r<\rho$  then  $ 2\le\cba\le (r/\ep)^\hal$ and for any $B\in  \Brep\cap\Barep$ such that $\dist(B,{\om_\ep}^c)>\ep$ we have
\begin{equation}\label{lamba}e_\ep(B)\ge 2\pi |d_B|\lapr,\quad\text{where}\quad \lapr = \hal \(\log\frac r{\ep\cba} - C_0\).\end{equation}
Moreover, $0\le \lapr\le \hal \lep$ and
\begin{equation}\label{deltalap} 0\le \hal\lep - \lapr\le\hal \(\beta\lep +|\log r| + C_0\).\end{equation}
\end{pro}

\begin{proof} From the definition \eqref{cba},
 from \eqref{bornenrj} and Lemma~\ref{lembornepetits}
  we have for  $\ep$  small enough  that
$2\le \cba \le \ep^{-\beta}$. It follows that
if $\ep^\hal<r<1$ then $ 2\le\cba\le (r/\ep)^\hal$, since
 $\beta<1/4$.  Also, from the definition of $\cba$ it is
 impossible that $e_\ep(\Barep\cap U_\a)\ge\cba\log(r/\ep)$
 since for $\sqrt{\ep} \le r\le \rho$ we have $\cba\ge
 3 e_\ep(\Barep)/\lep$.

Then from Proposition~\ref{jerr},  letting  $\overline C = \cba$, we deduce   \eqref{lamba} for any $B\in  \Barep$ such that $\dist(B,{U_\a}^c)>\ep$, which is equivalent to
$\dist(B,{\om_\ep}^c>\ep)$ if $B\in  \Brep\cap\Barep$.

Finally, $r/(\ep\cba)\ge \ep^{-\frac14}$ using $\cba\le
(r/\ep)^\hal$ and $r\ge \sqrt\ep$, which easily implies that
$\lapr>0$ if $\ep$ is small enough, and $\lapr \le \hal\lep$ is
clear from the definition. Inequality  \eqref{deltalap} follows from
$$ \hal\lep - \lapr = \hal\(\log\frac{\cba}r + C_0\)$$
since $\cba\le \ep^{-\beta}$.
\end{proof}

\section{Mass Transport}

We proceed to the  displacement of the negative part of
$$f_\ep = e_\ep -   \dl\nuep.$$

\subsection{Mass transport abstract lemmas}
For the displacements we will use the following two lemmas. The first,  more sophisticated one, was already stated in the introduction and  uses optimal transportation for the $1$-Wasserstein distance (or minimal connection cost).

\begin{lem}\label{transport} Assume $f$ is a finite  Radon measure on a compact set $A$,  that $\om$ is open  and   that for any positive  Lipschitz function  $\xi$ in $\lipom(A)$, i.e. vanishing on $\om\sm A$,
$$\int \xi \,df\ge - C_0 |\nab\xi|_{L^\infty(A)}.$$
Then there exists a Radon measure  $g$ on $A$ such that  $0 \le g \le f_+ $  and such that
$$\|f-g\|_{\lipom(A)^*}\le C_0.$$
\end{lem}
\begin{proof}
The proof uses convex analysis. Let $X= C(A)$ denotes the space of  continuous functions and for $\xi\in X$  let
$$
 \vp(\xi)= \int
\xi_+ \, df_+ \quad\text{and}\quad \psi(\xi ) = \begin{cases} +
\infty  & \text{if $|\nab \xi|_{L^\infty(A)}
> 1$ or  $\xi\notin\lipom(A)$}\\
- \int  \xi df  & \text{otherwise}\end{cases} .$$ Then $\psi$ is
lower semicontinuous because $\{\xi\in\lipom(A)\mid |\nab \xi|_{L^\infty}\le
1\}$ is closed under uniform convergence, and $\vp$ is continuous.
Moreover both functions are convex, and finite for $\xi = 0$. Then
the theorem of Fenchel-Rockafellar (see for instance \cite{et})
yields
$$\inf_X\(\vp+\psi\) = \max_{\mu\in X^*}\(-\vp^*(-\mu) - \psi^*(\mu)\),$$
where $X^*$ is the dual of $X$, i.e. the Radon measures on $A$  and
$$\vp^*(\mu) = \sup_{\xi\in X}\int \xi\,d\mu - \int \xi_+\,df_+ =  \begin{cases} 0 & \text{if } 0\le\mu\le f_+
  \\ +\infty  & \text{otherwise}\end{cases}, $$
$$\psi^*(\mu) = \sup_{\substack{\xi\in \lipom\\ |\nab\xi|_\infty\le 1}} \int\xi \,d\mu + \int \xi \,df =  \|\mu + f\|_{\lipom^*}. $$
We deduce that
$$\inf_{\substack{\xi\in \lipom\\ |\nab \xi|_{L^\infty}\le 1}} \int \xi_+\,df_+ - \int \xi\,df= \max_{0\le -\mu\le f_+}\( - \|\mu + f\|_{\lipom^*}\)$$
and then the existence of a Radon measure $g$ such that $-g$ maximizes the right-hand side, ie such that $0\le g\le f_+$ and
$$ - \|f-g\|_{\lipom^*} = \inf_{\substack{\xi\in \lipom\\ |\nab \xi|_{L^\infty}\le 1}} \int \xi_+\,df_+ - \int \xi\,df.$$ But
\begin{eqnarray*}
\inf_{ \substack{\xi\in \lipom\\ |\nab \xi|_{L^\infty}\le 1}} \int \xi_+ \, d f_+
- \int \xi
\, d f & = &  -\sup_{\substack{\xi\in \lipom\\ |\nab \xi|_{L^\infty}\le 1}} \( \int \xi \, d f
- \int \xi_+
\, d f_+\) \\
& = & - \sup_{\substack{\xi\in \lipom\\ |\nab \xi|_{L^\infty}\le 1}}
 \(\int \xi_+ \,d( f- f_+)- \int  \xi_-\,df \) \\& = & - \sup_{\substack{\xi\in \lipom\\ |\nab \xi|_{L^\infty}\le 1}}    \( - \int\xi_-\,df\) = \inf_{\substack{\xi\in \lipom\\ |\nab \xi|_{L^\infty}\le 1}}  \int \xi_- \, df .
\end{eqnarray*}
The assumption of the lemma implies
 that this last right-hand side is $\ge - C_0 $ therefore
$$\|f-g\|_{\lipom(A)^*}  \le C_0.$$
\end{proof}

The second,  less sophisticated, displacement result is
\begin{lem} \label{easytransport}
Assume $f$ is a finite Radon measure supported in $ \om$ and such
 that $f(\om)\ge 0$. Then there exists $0\le g\le f_+$ such that
  for any Lipschitz function $\xi$ $$\int_\om \xi\,d(f-g) \le 2
  \diam(\om) |\nab \xi|_{L^\infty(\om)}
 f_-(\om).$$
 \end{lem}
\begin{proof} This follows from the previous Lemma but can
 be proved directly by letting (assuming $f\neq 0$,
 otherwise  $g = 0$ is the answer),
$$ g = f_+\(1 - \frac{f_-(\om)}{f_+(\om)}\).$$
Then $g$ is positive because $f(\om)\ge 0$ implies $f_-(\om)\le f_+(\om)$ and
$$\int\xi\,d(f-g) = \int\xi\,d\(f_+  \frac{f_-(\om)}{f_+(\om)} - f_-\) = \int (\xi - \overline\xi)\,d\(f_+  \frac{f_-(\om)}{f_+(\om)} - f_-\),$$
 where $\overline\xi$ is the average of $\xi$ over $\om$,
 and the right-hand side is clearly  bounded above by
 $2\diam(\om) |\nab \xi|_\infty
 f_-(\om)$.
\end{proof}

\subsection{Mass displacement in the balls}
\begin{defi}
For  $B\in\Bep\cap\Baep$. We let
$$\fepB = (e_\ep - \Lap\nuep)\indic_{B\cap \om_\ep} .$$
where $\lapr$ is defined in \eqref{lamba} and we have set $\Lap = \laprho$.
\end{defi}
This corresponds to the excess energy in the balls i.e. the energy
remaining after subtracting off the expected value from the ball
construction. There is a difference of order $|\nuep|(B) \log \cba$ between
$f_\ep(B)$ and $\fepB(B)$  which will be dealt with
later.
\begin{pro}\label{lemgood}There exists $\ep_0,C>0$ such that for any $\ep<\ep_0$, and  any $B\in \Bep\cap\Baep$, there exists a positive measure
  $\gepB$ defined in $B\cap \om_\ep$ and such that
\begin{equation}\label{appxi}\gepB\le e_\ep +\Lap (\nuep)_-
\quad\text{and}\quad  \int_{B\cap \om_\ep}\xi\,d(\fepB - \gepB)
 \le C|\nab\xi|_{L^\infty(B\cap \om_\ep)} |\nuep |(B),
 \end{equation}
for any Lipschitz function $\xi$ vanishing on $\om_\ep\sm B$.
\end{pro}

\begin{proof}
To prove the existence of $\gepB$, in view of Lemma \ref{transport} and since $(\fepB)_+ = e_\ep + \Lap (\nu_\ep)_-$ on $B$
it suffices to prove that for any {\em positive} function $\xi$  defined on $B$ and vanishing on $B\sm \om_\ep$ we
have
\begin{equation}\label{dur}\int \xi\,d\fepB \ge - C|\nab\xi|_{L^\infty(B)}|\nuep |(B).\end{equation}
We turn to the proof of \eqref{dur}. Let $B\in \Bep\cap\Baep$ and $\xi$ be as above. Then
\begin{equation}\label{machin}\int \xi\,d\fepB  = \int_0^{+\infty} \fepB\(E_t\cap B\)\,dt,\end{equation}
where we have set $E_t=\{x\in B\mid \xi(x)\ge t\}$ and $\fepB(A) = \int_A \fepB.$

We will divide the integral \eqref{machin} into
$\int_0^{t_\ep}+\int_{t_\ep}^{+\infty}$,
 with $t_\ep = \ep |\nab\xi|_{L^\infty}$.
 The first integral is straightforward to bound from below.  Indeed $(\fepB)_-(B)\le C\lep|\nuep|(B)$ hence
\begin{equation}\label{morceaupetit}  \int_0^{t_\ep} \fepB\(E_t\)\,dt\ge - C \ep\lep |\nab\xi|_{L^\infty} |\nuep |(B)\ge - C |\nab\xi|_{L^\infty} |\nuep |(B).\end{equation}

On the other hand, if $t>t_\ep$, and this motivated our choice of $t_\ep$, then since $\xi = 0$ in $B\sm\om_\ep$ we have  $\dist(E_t,{\om_\ep}^c)>\ep$. Let then $t>t_\ep$, and $a\in E_t$ be a point in the support of $\nuep $. Then for
any $r\in[\sqrt\ep,\rho]$, there exists a ball $B_{a,r}\in\Brep$
containing $a$. Since $\{\Brep\}$ is monotonic with respect to $r$, $B_{a,r}\subset B$. We call $$r(a,t)=\sup
\{ r\in [\sqrt\ep,\rho), B_{a,r}\subset E_t\}$$ if this set is nonempty, and $0$ otherwise.  We then let
$$B_a^t = B_{a,r(a,t)}.$$

If $0<r(a,t)<\rho$ then $r(a,t)$ bounds from above the distance of $a$ to the complement of $E_t$. In particular
\begin{equation}\label{boundr} \xi(a) - t\le
r(a,t)|\nab\xi|_{L^\infty}.\end{equation} Indeed for any
$r(a,t)<s<\rho$ we have $B_{a,s}\subset B$ and $B_{a,s}\cap(E_t)^c
\neq\varnothing$ hence  there exists $b\in B_{a,s}\cap\partial E_t$.
Then  $\xi(a) - \xi(b)\le s |\nab\xi|_{L^\infty}$ and since $\p
E_t\subset \{\xi = t\}$ we deduce $\xi(a) -t\le s
|\nab\xi|_{L^\infty}$, proving \eqref{boundr} by making $s$ tend to
$r(a,t)$ from above.

A second fact is that if $r(a,t) = 0$, then
$\overline{B_{a,\sqrt\ep}}$ intersects $B\sm E_t$
and as above we deduce
\begin{equation}\label{boundrr}\xi(a) - t \le \sqrt\ep|\nab\xi|_{L^\infty(B)}.\end{equation}

The third fact is that   the collection
$\{B_a^t\}_a$, where $a$ ranges over $E_t$ and the $a$'s for
which $r(a,t) = 0$ have been excluded, is disjoint. Indeed take
$a,b\in E_t$ and assume that $r(a,t)\ge r(b,t)$. Then, since
$\B_{r(a,t)}$ is disjoint, the balls $B_{a,r(a,t)}$ and
$B_{b,r(a,t)}$ are either equal or disjoint. If they are disjoint we
note that $r(a,t)\ge r(b,t)$ implies that $B_{b,r(b,t)}\subset
B_{b,r(a,t)}$ and therefore $B_b^t = B_{b,r(b,t)}$ and $B_a^t =
B_{a,r(a,t)}$ are disjoint. If they are equal, then
$B_{b,r(a,t)}\subset E_t$ and therefore $r(b,t)\ge r(a,t)$,
which implies $r(b,t) = r(a,t)$ and then $B_b^t = B_a^t$.

Now, for any $B'\in \{B_a^t\}_a$  we have $B'\subset E_t$ and
$\dist(E_t,{\om_\ep}^c)>\ep$ hence $\dist(B',{\om_\ep}^c)>\ep$ and
from Proposition~\ref{goodbad},  we have, since $\lapr = \laprho
-\hal \log\frac\rho r$,
$$e_\ep(B')\ge |\nuep (B')|\(\Lap - \hal \log\frac\rho r\)_+ ,$$
where $r$ is the common value of $r(a,t)$ for $a$'s in $B'$ which
are in the support of $\nuep $. We may rewrite the above as
$$e_\ep(B')\ge \left|\sum_{a\in B'\cap\supp\nuep } \nuep (a)
 \(\Lap - \hal \log\frac\rho{r(a,t)}\)_+ \right|,$$
and summing over $B'\in \{B_a^t\}_a$ we deduce
$$e_\ep(E_t\cap B)\ge \left|\sum_{a\in \P_t} \(\Lap - \hal\log\frac\rho{r(a,t)}\)_+ \nuep (a)\right|,$$
where $\P_t$ is the set of points in $E_t\cap\supp\nuep $ such
that $r(a,t)>0$. We will let $\Q_t$
be the set of points in $E_t\cap\supp\nuep $ such that
$r(a,t)=0$.

Since $\nuep (E_t) = \nuep (\P_t) + \nuep (\Q_t)$, subtracting from
the above $\Lap \nuep (E_t)$  we find
$$\fepB\(E_t\)  \ge -\sum_{a\in \Q_t} |\nuep |(a) \Lap-
\hal \sum_{a\in \P_t} |\nuep |(a) \log\frac\rho{r(a,t)}.$$
 From \eqref{boundrr}, a given $a\in \supp\nuep \cap B$ can belong to $\Q_t$ only if $|t-\xi(a)|\le \sqrt\ep|\nab\xi|_{L^\infty}$. Therefore  integrating  the above with respect to $t$ yields,  using the fact that $t\le \xi(a)$ if $a\in E_t$, that
 $$ \int_{t_\ep}^{\infty} \fepB(E_t)\,dt \ge -
  \sum_{a\in \supp\nuep \cap B}|\nu_\ep|(a)
  \(\int_{\xi(a)- \sqrt{\ep} |\nab \xi|_{L^\infty}}^{\xi(a)+
   \sqrt{\ep} |\nab \xi|_{L^\infty}}  \Lambda_\ep^\alpha\,dt+
   \hal\int_0^{\xi(a)}\(\log\frac\rho{r(a,t)}\)_+\,dt\)$$ hence
 $$\int_{t_\ep}^{\infty} \fepB(E_t)\,dt\ge  -
 2\Lap \sqrt\ep |\nab\xi|_{L^\infty}|\nuep |(B) -
 \hal\sum_{a\in \supp\nuep \cap B} |\nuep |(a)\int_0^{\xi(a)}
 \(\log\frac\rho{r(a,t)}\)_+\,dt.$$
We now note that --- since $\Lap\le \hal\lep$ --- $\sqrt\ep\Lap$ is bounded independently of $\ep\le
1$ and, using the inequality \eqref{boundr}, we get
$$\int_0^{\xi(a)}\(\log\frac\rho{r(a,t)}\)_+\,dt\le \int_0^{\xi(a)}\(\log\frac{\rho|\nab\xi|_{L^\infty}}{\xi(a) - t}\)_+\,dt = \int_{\xi(a) - \rho|\nab\xi|_{L^\infty}}^{\xi(a)}\log\frac{\rho|\nab\xi|_{L^\infty}}{\xi(a) - t}\,dt, $$
and the rightmost integral is equal, by change of variables
$u=\frac{\xi(a)-t}{\ro |\nab \xi|_{L^\infty} } $,
 to $\rho |\nab\xi|_{L^\infty}$. Therefore
$$ \int_{t_\ep}^{+\infty} \fepB(E_t)\,dt \ge - C |\nuep |(B)|\nab\xi|_{L^\infty}.$$
In view of \eqref{machin}, adding \eqref{morceaupetit} yields the result.

\end{proof}

\begin{remark}\label{rqdurbis} Note that in  the proof of \eqref{dur}, the final radius $\rho$ may be replaced by any $r\in(\sqrt\ep,\rho)$. This yields the  following result: Assume that   $r\in(\sqrt\ep,\rho)$ and that $B\in\Brep$ is included in some ball in $\Bep\cap\Baep$. Then, for any positive function $\xi$ vanishing on $B\sm\om_\ep$,
\begin{equation}\label{durbis}\int_B (e_\ep -  \lapr\nuep)\xi\ge - C|\nab\xi|_{L^\infty(B)}|\nuep |(B).\end{equation}
\end{remark}
We record the following lower bounds:

\begin{pro}\label{lowb} For $\ep$ small enough and $B\in\Bep\cap\Baep$:
\begin{equation}\label{prolb1}
e_\ep(\om_\ep\cap B) \ge  \(\frac18\lep-C\)|\nuep
|(B).\end{equation} For $\ep $ small enough and $B\in \Bep \cap
\Baep$ such that $\dist(B, \om_\ep^c)>\ep$, we have
\begin{equation} \label{prolb2}
 \gepB(\om_\ep\cap B) \ge \(\frac18\lep-C\)|\nuep |(B) - \dl |\nuep (B)|.\end{equation}
If in addition  $d_B<0$, then
\begin{equation}\label{bipmoins}  \gepB(\om_\ep\cap B) - \(\dl - \Lap\)
\nuep(B)  \ge \(\frac{1}{8}\lep - C\) |\nuep|(B).\end{equation}
\end{pro}

The meaning of this lower bound is that $e_\ep(B)$ is not only
bounded below by $\Lap|\nuep(B)|$, which to leading order is
$\hal\lep|\nuep(B)|$ --- this is the positivity of $\gepB$ in the
above proposition --- but also by some constant times
$\lep|\nuep|(B)$, even though the constant is no longer guaranteed to
be the (optimal) value $1/2$. This information is valuable in the
case where $|\nuep(B)|$ is much smaller than $|\nuep|(B)$. The
precise value of the constants is unimportant.
\begin{proof}
 As we noticed, $\cba<(\sqrt\ep/\ep)^\hal$ implies   $\sqrt\ep/(\ep\cba)\ge \ep^{-\frac14}$  thus,   using  Proposition~\ref{goodbad},
\begin{multline*}e_\ep(B\cap\om_\ep)\ge \sum_{\substack{B'\in\Bsep\\B'\subset B\\\dist(B',{\om_\ep}^c)>\ep}} e_\ep(B')\ge \sum_{\substack{B'\in\Bsep\\B'\subset B\\\dist(B',{\om_\ep}^c)>\ep}} \pi|d_{B'}|\(\log\ep^{-\frac14} - C\) = \\
 |\nuep|(B)\(\frac18 \lep - \frac{C}{2}\),\end{multline*} which proves the first
 assertion.
Secondly, note  that from \eqref{appxi}, if $\dist (B,
\om_\ep^c)>\ep$, choosing $\xi$ compactly supported in $\om_\ep$
such that $\xi = 1$ in $B$, we have
$$\fepB(B\cap\om_\ep) =  \gepB(B\cap\om_\ep).$$
Since $\Lap\le \hal\lep$ we deduce \eqref{prolb2} in view of
$$ \gepB(B\cap\om_\ep) = \fepB(B\cap\om_\ep) \ge  |\nuep|(B)
\(\frac18 \lep - C\) - \hal\lep |\nuep(B)|. $$ For the last
assertion, since $\nuep(B) = 2\pi d_B <0$, we write
$$\gepB(B\cap\om_\ep) - \(\hal\lep - \Lap\)\nuep(B) = e_\ep(B\cap\om_\ep) - \hal\lep\nuep(B)\ge e_\ep(B\cap\om_\ep),$$
and this is bounded below using \eqref{prolb1}.
\end{proof}

\subsection{Mass displacement of the remainder}
Proposition~\ref{lemgood} will allow to replace $\fepB$ by the positive $\gepB$, and we have
\begin{equation}\label{sto}\fep - \sum_{B\in\Bep} \fepB = e_\ep\indic_{\Bep^c} +
\sum_{\a}  \(\dl - \Lap\) \nua.\end{equation}
We now proceed to absorb the
negative part of $\fep - \sum \fepB$, which is  $\(\dl - \Lap\) (\nua)_+$. This will be easy if $\cba =\frac{3 \eaep}{\lep}$
and
if not, in view of  \eqref{lamba},  we have
$$ 0\le \dl - \Lap \le \hal\log\naep +C,$$
 which allows to bound the mass of the negative part by $C\sum_\a  \naep(\log\naep +1)$. Following the method in \cite{ss0}  (see also \cite{livre}, Chap.~9), this will be balanced by a lower bound by $c [\naep]^2$
  for the energy on annuli surrounding $U_\a$.

Recall that $U_\alpha = B(x_\alpha,\ell_0)\cap \om_\ep$. We let
$A_\a = B(x_\a,r_1)$, where   $r_1 = 3\ell_0$. Choosing $\ell_0$
small enough, we may require that
$$ \diam(A_\a) < 1\quad\text{and}\quad \(A_\a\cap{\om_\ep}^c\neq\varnothing\implies A_\a\subset \left\{x\mid \dist(x,\partial\om_\ep)<\hal\right\}\).$$
We will denote below by $m'$ a bound, uniform in $\ep$ for the
overlap number  of  the $\{A_\a\}_\a$.

Now we choose $\rho$ such that for any $\ep>0$
$$ \mbox{  $|\Tae|\ge \ell_0,$ where  $\Tae$ is the set of
$t\in(r_0,r_1)$  such that $\{|x-x_\alpha|=t\}\cap\Bep =
\varnothing$},$$ where $r_0=\ell_0$.  Indeed, the number of
$U_\beta$'s which intersect $B(x_\alpha,r_1)$ is bounded by a
certain number $N$ independent of $\ep$ and $\alpha$. Choosing
$\rho=\ell_0 / N$, the sum of the radii of balls in
$\cup_\beta\Bbep$ which intersect $B(x_\alpha,r_1)$ is bounded above
by $\ell_0$, hence $|\Tae|\ge (r_1-r_0)-\ell_0 = \ell_0$.

\subsubsection{Lower bounds on annuli}
For any  $\a$ let
\begin{equation}\label{gepal}
\gepp  =  \frac{1}{4m'}\(e_\ep\indic_{\Bep^c} +\sum_{ B\in  \Bep}  \gepB\) \indic_{A_\a}, \quad    \gepm =  \(\dl - \Lap\) (\nuep)_+\indic_{\Bep\cap\Baep},\end{equation}
and $\gepal = \gepp - \gepm$. We have
$$\tgep - \sum_{\a} \gepal \ge \frac{3}{4} \(e_\ep\indic_{\Bep^c} +\sum_{ B\in  \Bep}  \gepB\) + \sum_\a \(\dl - \Lap\) (\nuep)_-\indic_{\Bep\cap\Baep}.$$
In particular
$$\gepp(A_\a)\le \frac{1}{3m'}\(\tgep - \sum_{\beta}\tilde g_\ep^\beta\)(A_\a).$$

\begin{pro} There exist  $\ep_0,C,c >0$ such that  if   $\beta<1/4$  in \eqref{bornenrj},  then  for any $\ep<\ep_0$ and any  index  $\a$
\begin{equation}\label{toujoursmoins} \gepm(A_\a) \le \pi\naep\(  \beta\lep + C\). \end{equation}
If moreover  $\dist(A_\a,{\om_\ep}^c)>\ep$ then at least one of the following is true:
\begin{equation}\label{mincalmoins} \gepm(A_\a) \le \pi\naep\(  \beta\lep + C\),\quad \gepp(A_\a)\ge c \naep\lep\end{equation}
or
\begin{equation}\label{mincal}  \gepm(A_\a) \le \pi\naep \(\log\naep + C\) ,\quad \gepp(A_\a)\ge c \naepp.\end{equation}
\end{pro}
\begin{proof}
The bound \eqref{toujoursmoins} follows from  \eqref{gepal}, \eqref{deltalap}. Now assume $\dist(A_\a,{\om_\ep}^c)>\ep$.

First, if $\naep = 0$ then $\gepm = 0$, $\gepp\ge 0$  hence \eqref{mincalmoins} is true.

Second, if $ 3 \eaep/\lep\ge M\naep$ then, since for $B \subset
A_\alpha$, we have  $\gepB(B) = \fepB(B) = e_\ep(B) - \Lap \nuep(B)$
and $\Lap\le \hal\lep$ it follows that
\begin{multline*}
\gepp(A_\a)\ge \frac{1}{4m'} \int_{A_\a}e_\ep- \frac{1}{4m'}\Lap \sum_{B\in \Bep \cap A_\a} |d_B| \ge \frac{1}{4m'} \int_{U_\a} e_\ep
- \frac{1}{4m'}\Lap \sum_{B\in \Bep \cap A_\a} |d_B|
\\
\ge \frac M{12 m'}\naep\lep -
\pi\naep\lep \ge \(\frac M{12 m'}-\pi\)\naep\lep.\end{multline*}Together with
\eqref{toujoursmoins}, this implies \eqref{mincalmoins} if $M$ was
chosen strictly  greater than $12 m'\pi$.
The last case is that where $\cba = M\naep$. Then $\dl - \Lap = \hal\log\naep + C$ and therefore, using \eqref{bornepetits},
\begin{equation}\label{gepm}\gepm(A_\a)\le 2\pi\naep\(\hal\log \naep+C\)\le \naep(\pi\beta\lep +C).\end{equation}
We define
$$D_0^+ = \sum_{\substack{B\in\Bep \\B\subset B(x_\alpha,r_0)\\d_B>0}} d_B,\quad D_1^- = \sum_{\substack{B\in\Bep\\B\subset B(x_\alpha,r_1)\\d_B<0}} |d_B|,$$
and again we distinguish several  cases.

First from \eqref{gepm} we will have proven \eqref{mincalmoins} if we prove that
\begin{equation}\label{minobeta} \gepp(A_\a)\ge c \naep\lep,\end{equation}
for some  $c>0$. This inequality holds in the following  two cases.

{\em First case : $D_1^- > \naep/20$.}   This means there is a
significant proportion of balls with negative degrees. For each such
negative ball  we have from \eqref{bipmoins}, and since
$|\nu_\ep|(B) \ge |\nu_\ep(B)|$,
$$ \gepB(B)\ge \gepB(B) - \(\dl - \Lap\)\nuep(B)  \ge   \(\frac18 \lep - C\)2\pi |d_B|.$$
This implies that
$$\gepp(A_\a)\ge \frac{1}{4 m'}\(\frac18 \lep - C\) 2\pi D_1^-,$$
hence \eqref{minobeta} is satisfied when $D_1^- > \naep/20$. \medskip

{\em Second case :  $D_0^+ \le  \naep/10$ and $D_1^- \le \naep/20$.}
Then for each $B\in\Bep\cap\Baep$, Proposition~\ref{lowb} yields
$$\gepB(B)\ge\begin{cases}
 \(\frac18 \lep - C\)|\nuep|(B) - \dl|\nuep(B)| & \text{if $|d_B|>0$}\\
\( \frac18\lep - C\)|\nuep|(B) & \text{if
$|d_B|<0$.}\end{cases}$$ Summing with respect to $B$ we
find, since $B\in \Bep\cap\Baep$ implies $B\subset
B(x_\alpha,r_0)$, that
$$\gepp(A_\a)\ge \frac{1}{4 m'}\( \frac18\lep - C\)\naep   -  \frac{1}{4 m'} D_0^+ \dl,$$
which again yields \eqref{minobeta} when $D_0^+ \le  \naep/10$.

We are left with the third case, when  $D_0^+ \ge \frac{\naep}{10}$
 and $D_1^-\le
\frac{\naep}{20}$. In this case \eqref{minobeta}
 and then \eqref{mincalmoins} do not necessarily hold.
  We need to prove \eqref{mincal} instead, which in view of
   \eqref{gepm} reduces to proving
$$\gepp(A_\a)\ge c \naepp.$$
For this  we really need to use the lower bounds on annuli of the type first introduced in \cite{ss0}.
 We denote
$$\anep = B(x_\alpha,r_1)\setminus \(B(x_\alpha,r_0)\cup\Bep\).$$
For any $t\in \Tae$ we let $B_t = B(x_\alpha,t)$ and
$\gamma_t=\partial B_t$ and recall that $\gamma_t$ does not
intersect $\Bep$. If $t\in \Tae$ then $|u_\ep|\ge 1/2$ on $\gamma_t$
because of \eqref{SE} and the fact that
$\dist(A_\a,{\om_\ep}^c)>\ep$.

It follows  (see for instance \cite{livre} Lemma 4.4, or
\eqref{lbcircle} below) that for some constant $c>0$ we have
\begin{equation}\label{cercle}\int_{\gamma_t}\( \hal |\nab_{A} u|^2
+ \frac{(1-|u|^2)^2}{4\ep^2}\) +\hal
 \int_{B_t} (\curl A)^2 \ge c \frac{|\dtep|^2}{t},\end{equation}
where $\dtep$ is the degree of $u_\ep/|u_\ep|$ on $\gamma_t$.
Integrating \eqref{cercle} with respect to $t\in\Tae$, which has measure less than $1$,  the left-hand
side will be bounded above by $e_\ep(A_\a)$.  In view of the lower bound $\dtep\ge
\(D_0^+ - D_1^-\)$, which is  valid for any $t\in\Tae$,  since $|\Tae|\ge
\ell_0$, and from the assumption on $D_0^+$ and $D_1^-$  we deduce that
$$e_\ep\(A_\a\sm\Bep\)\ge c\(D_0^+ - D_1^-\)^2\ge c\naepp.$$
Then, since $\gepp = \frac1{4m'} e_\ep$ on $(\Bep)^c$ we deduce $\gepp(A_\a)\ge c\naepp$ and \eqref{mincal} is proved.
\end{proof}

\subsection{Proof of Theorem~\ref{th1} and Corollary~\ref{coro1}} \mbox{}

\noindent{\em Item 1).} The estimate \eqref{estjac} was already mentioned after the definition \eqref{nu} of $\nuep$, and the bound $|\nuep|(E)\le C e_\ep(\HE)/\lep$ was proved in Lemma~\ref{lembornepetits}. \medskip

\noindent{\em Item 2).} We define
$$\fa = \sum_{B\in\Bep\cap\Baep} (\fepB - \gepB) + \gepp - \gepm. $$
Then clearly $\fa$ is supported in $A_\a$. Moreover, using the fact (see \eqref{sto}) that
$$\fep - \sum_{B\in\Bep}\fepB = e_\ep\indic_{\Bep^c} - \sum_\a\(\dl - \Lap\)\nua$$
and since $\sum_\a\indic_{A_\a}\le m'$ we easily obtain
\begin{equation}\label{eqfep} \fep - \sum_\a\fa = \sum_\a\(\dl - \Lap\)(\nua)_- + \(e_\ep\indic_{\Bep^c} + \sum_{B\in\Bep}\gepB\)\(1 - \frac1{4m'}\sum_\a\indic_{A_\a}\). \end{equation}
Since $\sum_\a\indic_{A_\a}\le m'$ we find
\begin{equation}\label{mff} \fep - \sum_\a\fa \ge \sum_\a\(\dl - \Lap\)(\nua)_- + \frac34\(e_\ep\indic_{\Bep^c} + \sum_{B\in\Bep}\gepB\)\ge 0. \end{equation}\medskip

\noindent{\em Item 3).} We define $\ga$. In the case $\dist(A_\a, {\om_\ep}^c)\le \ep$   we let $\ga = \gepp$.
Then
$$\int \xi\,d(\fa-\ga) = \sum_{B\in\Bep\cap\Baep}\int\xi\,d (\fepB - \gepB) - \int\xi\,d\gepm.$$
This implies \eqref{casbord}, summing \eqref{appxi} over $B\in \Bep\cap\Baep$ and using \eqref{toujoursmoins}.

In the case $\dist(A_\a, {\om_\ep}^c)>\ep$  we let
$$c_\a =\( \frac{\gepal(A_\a)}{|A_\a|}\)_-. $$
We deduce easily from  \eqref{mincalmoins}, \eqref{mincal} and if
$\beta$ is small enough  that  $c_\a\le C$ and applying
Lemma~\ref{easytransport} in $A_\a$ to $\gepal+c_\a$ we obtain
$\vp_\a$ defined on $A_\a$ and such that $0\le \vp_\a\le \gepp +
c_\a$ and,  for any Lipschitz function $\xi$,
$$\int_{A_\a} \xi\,d\(\gepal -\ga\) \le  C|\nab\xi|_{L^\infty(A_\a)} \gepm(A_\a),\quad\text{where $\ga: = \vp_\a-c_\a$.}$$
Moreover $-C\le -c_\a\le \ga\le\gepp$.

Then
\begin{equation}\label{faga}\begin{split}\int_{A_\a} \xi\,d(\fa-\ga)
&= \int_{A_\a} \xi\,d(\fa-\gepal)+\int_{A_\a} \xi\,d(\gepal-\ga)\\
&= \sum_{B\in\Bep\cap\Baep}\int\xi\,d (\fepB - \gepB) + \int_{A_\a} \xi\,d(\gepal-\ga)\\
&\le C |\nab\xi|_{L^\infty(A_\a)}\( \naep + \gepm(A_\a)\),
\end{split}\end{equation}
where we have used \eqref{appxi} to bound the integral involving $\fepB-\gepB$. Moreover, $\ga(A_\a)= \gepal(A_\a)$.

If \eqref{mincalmoins} holds, then \eqref{cas1} follows immediately from \eqref{faga} when $\pi\beta < c/2$, with $c$ the constant in \eqref{mincalmoins}. If \eqref{mincal} holds  we deduce \eqref{cas2} from \eqref{faga} by noting that $c \naep^2 - C\naep(\log \naep +1) \ge \frac{c}{2}\naep^2 - C'\naep$ if $C'$ is chosen large enough depending on $c,C$.

\medskip

\noindent {\em Item 4), \eqref{prosurj}.}
 We   adapt an argument in \cite{struwe}.

 First, $g_\ep - \sum_\a \ga  = \fep - \sum_\a\fa$ thus from \eqref{mff} and since $\sum_\a\ga\ge -C$ we find
\begin{equation}\label{mgg} g_\ep\ge \frac34\(e_\ep\indic_{\Bep^c} + \sum_{B\in\Bep}\gepB\) - C.\end{equation}
 Then, assuming $U_\a\subset\om_\ep$, denote by $\Brepa$ the set of balls in $\Brep$ which are included in some ball belonging to $\Baep\cap\Bep$, so that $\nua(\Bep) = \nuep(\Baep\cap\Bep) = \nuep(\Brepa)$. Applying  Remark~\ref{rqdurbis} for some $r\in(\sqrt\ep,\rho)$ with $\xi=1$ and  summing \eqref{durbis} over $B\in\Brepa$ we find  $ e_\ep(\Brepa)\ge \lapr\nuep(\Brepa)$ and then
\begin{equation*}\begin{split} e_\ep(\Bep\cap\Baep\sm\Brepa)&\le e_\ep(\Bep\cap\Baep) - \Lap \nua(\Bep) + \(\Lap-\lapr\)\nua(\Bep)\\
& = \sum_{B\in\Bep\cap\Baep} \gepB(B) + \hal\log\frac1r \nua(\Bep),
\end{split}\end{equation*}
where we have used the fact that $\fepB(B) = \gepB(B)$. It follows using \eqref{mgg} that
\begin{equation}\label{comple}e_\ep(\Bep\cap\Baep\sm\Brepa)\le C\(\gp(U_\a) + \naep\log\frac1r +1\).\end{equation}

Then comes the argument in \cite{struwe}: For any integer $k$,
 let $r_k=2^{-k}\rho$, and let $\C_k$ be the intersection of
  $\Bkep\sm\Bkiep$ and  $\Baep$. Then $|\C_k|\le C 2^{-2k}
  \rho^2$, since $\rho 2^{-k}$ bounds the total radius of the balls in $\Bkep\cap\Baep$. Moreover $j_\ep = (iu_\ep,\nab u_\ep-iA_\ep)$ and thus assuming $|u_\ep|\le 1$ we have $|j_\ep|^2\le 2 e_\ep$. Then using H\"older's inequality in $\C_k$ and (\ref{comple})
   we find for $p<2$
\begin{equation*}\begin{split}\int_{\C_k}|j_\ep|^p & \le
 |\C_k|^{1-p/2}\(e_\ep(\C_k)\)^{p/2} \le
 |\C_k|^{1-p/2}  \(e_\ep(\Bep\cap\Baep\sm\Bkiep)\)^{p/2}\\
& \le  C_p 2^{-(2-p)k}  \( e_\ep(\Bep\cap\Baep\sm\Bkiep)\)^{p/2}\\
& \le C_p 2^{-(2-p)k} \(\gp(U_\a) + k\naep\log 2 +1 \)^{p/2}.
\end{split}\end{equation*}
Using \eqref{controlenu} we find
$$ \int_{\C_k}|j_\ep|^p \le C_p 2^{-(2-p)k}\(1+k\log2\)^{p/2} \(\gp(U_\a) +1 \)^{p/2}.$$
Summing these inequalities for $k$ ranging from $0$ to the largest
integer $K$ such that $r_K\ge\sqrt\ep$ --- so that  in particular $r_K\le
2\sqrt\ep$ --- we find
$$\int_{\Bep\cap\Baep\sm\Bdsep}|j_\ep|^p
 \le   C_p \( \gp(U_\a) +  1\)^{p/2} ,$$
where $C_p$ is a constant times  the sum of the convergent series $\sum_k 2^{-(2-p) k} (1+ k\log2 - \log\rho)^{p/2}$.  To this inequality we add
$$\int_{\Baep\cap\Bdsep}|j_\ep|^p\le
C\ep^{1-p/2} e_\ep(U_\a)^{p/2},$$
which follows from H\"older's inequality after estimating as above $|\Bdsep\cap\Baep|$ by $C\ep$. But  since $e_\ep= f_\ep  +
 \hal \lep\nu_\ep$ we may write using \eqref{err}, \eqref{controlenu},
 \begin{equation}\label{c4}
 e_\ep(U_\a)\le C \gp(\widehat U_\a) + C |\nu_\ep|(\widehat U_\a) (1+ \lep)\le C\lep\(\gp(U_\a+B(0,2))+ 1\).
 \end{equation}
Thus
\begin{equation*}\begin{split}\int_{\Baep\cap\Bdsep}|j_\ep|^p&\le
C \ep^{1-\frac p2} \lep^{\frac p2} \(
\gp(U_\a+B(0,2))^{\frac p2} +1\) \\ &\le C\(
\gp(U_\a+B(0,2))^{\frac p2} +1\). \end{split}\end{equation*}
We also add
$$\int_{U_\a\sm\Bep}|j_\ep|^p\le
 C\(\gp(U_\a) + 1\)$$
which follows from \eqref{mgg}.
 Finally we obtain
\begin{equation*}\int_{U_\a}  |j_\ep|^p\le C_p \( \gp(U_\a+B(0,2)) +  1\).
\end{equation*}
Summing with respect to the $\a$'s such that $E\cap U_\a\neq\varnothing$, this proves \eqref{prosurj} and concludes the proof of Theorem~\ref{th1}.

\noindent{\em Proof of Corollary~\ref{coro1}.} Note that
$$\int \xi\,d(f_\ep - g_\ep)    = \sum_\a \int \xi \,d(\fa-\ga).$$
Three types of indices occur.

First we consider indices $\a$ such that  $\dist(A_\a, {\om_\ep}^c)>\ep$ and \eqref{cas1} holds. Since  \begin{equation}\label{gage}\ga \le g_\ep - \sum_{\beta\neq\a} g_\beta \le g_\ep + C,\end{equation}
we deduce  from \eqref{cas1} that if $\naep\ge 1$ and $\ep$ is small enough, $g_\ep(A_\a)\ge c\naep\lep$  and then using \eqref{cas1} again that
\begin{equation}\label{faga1}
\int \xi \,d(\fa-\ga) \le C |\nab\xi|_{L^\infty(A_\a)} \(\naep + \beta\gp(A_\a)\). \end{equation}
If $\naep = 0$ the same inequality holds since from \eqref{cas1} the left-hand side is zero.

Second we consider indices $\a$ such that $\dist(A_\a, {\om_\ep}^c)>\ep$ and \eqref{cas2} holds.
We note that if $C$ is large enough then  $x\log x\le \eta x^2
 + C\log ^2\eta/\eta$ holds  for every $x>0$ and $\eta\le 1$,
 for instance by distinguishing the cases $\eta>\frac{\log x}{x}$
  and  $\eta\le \frac{\log x}{x}$. We use this and  \eqref{gage}, together with \eqref{cas2} to find that if $\naep\ge 1$ then
\begin{equation}\label{faga2}
\int \xi \,d(\fa-\ga) \le C |\nab\xi|_{L^\infty(A_\a)} \(\naep + \eta \gp(A_\a)+ \frac{\log^2\eta}{\eta}\). \end{equation}
Again the inequality is true if $\naep = 0$ since from \eqref{cas2} the left-hand side is zero in this case.

Finally we consider indices $\a$ such that $\dist(A_\a, {\om_\ep}^c)\le \ep$. In this case, noting that from Lemma~\ref{lembornepetits} we have $\naep\lep\le C e_\ep(A_\a)$, we rewrite \eqref{casbord}  as
\begin{equation}\label{faga3}
\int \xi \,d(\fa-\ga) \le C \(|\nab\xi|_{L^\infty(A_\a)} \naep + \beta |\xi|_{L^\infty(A_\a)} e_\ep(A_\a)\). \end{equation}

To conclude we sum either \eqref{faga1}, \eqref{faga2} or \eqref{faga3} according to the type of index $\a$, noting  that since $\diam(A_\a)\le 1$, we have $|f|_{L^\infty(A_\a)}\le \widehat f$ on $A_\a$ for any function $f$. Since the overlap number of the $A_\a$'s is bounded by a universal constant, we deduce \eqref{err}.

We prove  \eqref{controlenu}. We  start by proving that when $\dist(A_\a, {\om_\ep}^c)>\ep$ we have
\begin{equation}
\label{340b} \min \( \naepp, \naep \lep\) \le C\(\gp(A_\a)+1\).
\end{equation}
If $\naep = 0$ this is trivial, if not then it follows from either \eqref{cas1} or \eqref{cas2} using \eqref{gage}.

Assume $\a$ is such that $\dist(A_\a, {\om_\ep}^c)>\ep$, then since $2x\le \eta  x^2 + 1/\eta$ and since $x\le \eta x \lep$ is trivially true if $1/\lep<\eta$, we deduce from \eqref{340b} that
\begin{equation}\label{nbin}\naep\le C\(\eta \gp(A_\a) + 1/\eta\).\end{equation}
On the other hand  Lemma~\ref{lembornepetits} implies that  for any $\a$
\begin{equation}\label{nbout}\naep\le C\frac{e_\ep(A_\a\cap\om_\ep)}{\lep}.\end{equation}
Summing \eqref{nbin} or \eqref{nbout} according to whether
$\dist(A_\a, {\om_\ep}^c)$ is $>\ep$ or $\le \ep$ we deduce  \eqref{controlenu}.

\section{Proof of Theorem~\ref{th2}}
\subsection{Convergence} We study the consequences of the
hypothesis
\begin{equation}\label{limlims}\forall R>0,\quad \GR := \limsup_{\ep\to 0}  \int_{\UR}g_\ep(x)\,dx<+\infty.\end{equation} and  prove that it implies the convergence of the vorticities and currents in the appropriate sense.

Note that we assume $\dist (0 , \p \om_\ep) \to + \infty$ so that for every $R$, $\mathbf{U}_R \subset \om_\ep$ for $\ep$ small enough.
From \eqref{hypsets} there exists $C>0$ such that for any $R$ large enough
$$ B_{R/C}\subset \UR\subset B_{CR},\quad \frac1C\le \frac{|\UR|}{R^2}\le C.$$
We now gather several easy  consequences of Theorem~\ref{th1} and (\ref{limlims}).

\begin{pro}
Assume \eqref{limlims} holds, and let $g_\ep$ be as in Theorem~\ref{th1}. Then for any $R$ and  $\ep$ small enough depending on $R$ we have
\begin{equation}\label{calbound} \sum_{\alpha\mid A_\alpha\subset \UR} \min(\naepp, \naep \lep)\le C (\GRC +R^2),\end{equation}
\begin{equation}\label{morebounds1}|\nuep|(\UR) \le C (\GRC +R^2),\end{equation}
\begin{equation}\label{morebounds2}\int(\fep-\gep)\chi_{\UR} \le  C  \sum_{\alpha\mid A_\alpha\subset \URC\sm \URc}
 \naep(\log\naep +1)  \le C (\GRC +R^2), \end{equation}
where $\{\chi_{\UR}\}_R$ are any functions satisfying
\eqref{defchi}.

For any $1\le p<2$ there exists $C_p>0$ such that  for any $R>0$, and $\ep$ small enough
\begin{equation}\label{jbound}\int_{\UR}|j_\ep|^p  \le C_p
(\GRC+ R^2).\end{equation} Moreover, up to extraction of a
subsequence, $\{j_\ep\}_\ep$ converges weakly in $L^p_{loc}(\mr^2)$,
$p<2$ to some $j : \mr^2 \to \mr^2$;
  $\{\nuep\}_\ep$  converges in the weak sense of measures to a measure
  $\nu$ on $\mr^2$ of the form $2\pi \sum_{p \in \Lambda} d_p
  \delta_p$ where $\Lambda$ is a discrete set and $d_p \in \mz$, $\{\mu_\ep\}_\ep$  converges to the same $\nu$ in $W^{-1,p}_\loc(\mr^2)$ for any $p<2$    and $\{h_\ep\}$ converges   weakly in $L^2_{loc}(\mr^2)$ to
   $h$. Moreover it holds that
\begin{equation}\label{curlj} \curl j = \nu -h.\end{equation}
\end{pro}

\begin{proof}
Assertions \eqref{calbound}, \eqref{morebounds1} and \eqref{jbound}  are  direct consequences of \eqref{340b},   \eqref{controlenu} and \eqref{prosurj}, respectively.

We prove  \eqref{morebounds2}. First we note that a consequence of  \eqref{limlims} is that for every $R>0$,  if $\ep>0$ is small enough and $A_\a\subset \UR$  then \eqref{cas2} holds. Indeed if \eqref{cas1} is true with $\naep \ge 1$ (note that if $\naep = 0$ then \eqref{cas1} and \eqref{cas2} are identical) then $g_\ep(A_\a)\ge c\lep-C$, using \eqref{gage}, which contradicts \eqref{limlims} if $\ep$ is small enough.

Then we use \eqref{cas1} with $\xi =  \chi_{\UR}$. From the fact that   $\chi_\UR$ is supported in $\URC$ and since $\dist(\URC, \p \om_\ep) \to + \infty$ we have, if $\ep$ is small enough and $A_\a\cap\URC\neq\varnothing$, that $\dist(A_\a, \p \om_\ep)>\ep$. Then summing \eqref{cas1} over all such $\a$ we find
$$ \int \chi_\UR\,d(f_\ep - g_\ep)
  \le C  \!\!\!\!\!\!
  \sum_{\substack{\text{$\alpha$ s.t.}\\ A_\a\subset \URC\sm\URc}}
  \!\!\!\!\!\!   \naep \(\log\naep + 1\),$$
which is  the first inequality in \eqref{morebounds2}. The second one then easily follows from \eqref{340b}.

We now turn  to the convergence  results.  The weak local  convergence  of $j_\ep$  follows from  a bound
for $\int_{\UR} |j_\ep|^p$ valid for any $\ep$ small enough, depending
on $R$, which is implied by (\ref{limlims}) and (\ref{jbound}).
 From
(\ref{morebounds1}), $\{\nu_\ep\}_\ep$ is bounded on any compact
subset of $\mr^2$, hence converges (up to extraction) to a measure
$\nu$, which  by (\ref{nu}) has to be of the form $2\pi \sum_{p \in
\Lambda} d_p \delta_{p}$ where $\Lambda $ is a discrete set and $d_p\in\mz$ for every $p\in\Lambda$ (we will
prove below that $d_p=1$).

The weak local convergence of $h_\ep$ follows from Remark
\ref{rem51} combined with the bound (\ref{limlims}).

The convergence of $\{\mu_\ep\}_\ep$ in $W^{-1,p}_\loc$  uses
    the Jacobian estimate (see \cite{js} or \cite{livre}, Theorem 6.2)
from which we deduce that for any $R>0$ and any $\gamma\in(0,1)$, and since $r(\Bsep\cap B_R)\le C \sqrt\ep$,
\begin{equation}\label{jacest}\|\mu_\ep - \nuep\|_{(C_0^{0,\gamma}(B_R))^*}
\le C(\sqrt\ep)^\gamma\(e_\ep(B_R)+1\),\end{equation}
where $C$ depends on $R$ but not on $\ep$.

But $\{\nuep\}_\ep $ is bounded in $B_R$ as measures,
 hence in $(C_0^{0,\gamma})^*$, and arguing again as in \eqref{c4},
$$e_\ep(B_R)\le  \gp(B_{R+1}) + \dl|\nuep|(B_{R+C})\le C\lep$$
therefore the right-hand side in \eqref{jacest} tends to $0$ as
$\ep\to 0$ and $\{\mu_\ep\}_\ep$ is bounded in
$(C_0^{0,\gamma}(B_R))^*$. We deduce that $\mu_\ep\to\nu$  in
$W^{-1,p}_\loc$ by noting that for any $1<p<2$ there exists
$0<\gamma<1$ such that $W^{1,p'}_0(B_R)\hookrightarrow
C^{0,\gamma}_0$ with compact imbedding --- where $1/p + 1/p' = 1$
--- which implies by duality that  $(C^{0,\gamma}_0)^*
\hookrightarrow W^{-1,p}_0$ with compact imbedding.

Finally  \eqref{curlj} is obtained by passing to the limit in
$\mu_\ep = \curl j_\ep + \curl A_\ep$
 since  by Remark \ref{rem3.2} we may assume  (up to extraction) that  $\curl
A_\ep\to h$ weakly locally in $L^2$ as $\ep \to 0$.
\end{proof}

\begin{remark} From the above results, it is easy to deduce
 \eqref{lpunif} under the stronger assumption \eqref{hypcoerplus}.
  In this case we have  $\GR\le CR^2$ and therefore
  \eqref{morebounds1}, \eqref{jbound} and  Remark \eqref{rem3.2}
   imply that
\begin{equation}\label{uniformes}|\nuep|(\UR)\le CR^2,\quad \int_{\UR}|j_\ep|^p\le CR^2 , \quad \int_{\UR} |h_\ep|^2 \le C R^2 \end{equation}
which in turn implies \eqref{lpunif}.
\end{remark}

\subsection{Lower bound by the  renormalized energy}

We turn to the proof of the remaining statement in Theorem~\ref{th2},
namely that $\nu$ is of the form $2\pi\sum_{p\in\Lambda}\delta_p$
(we already know it is of the form $2\pi\sum_{p\in\Lambda}d_p
\delta_p$, where the $d_p$'s are nonzero integers) and that under
assumption \eqref{hypcoerplus} the lower bound \eqref{limlim} holds.
Both are related to a lower bound of $\int \chi_R \gep$ by the
renormalized energy, where $\chi_R:=\chi_{\UR}$. This reproduces
more or less  arguments present  in \cite{bbh} and \cite{betriv}.
Throughout this subsection we assume that \eqref{hypcoerplus} holds,
and begin by bounding from below the integral of
$(e_\ep-\dl\nuep)\chi_R$.

Choose $R>0$. From \eqref{morebounds1} we have
 that $|\nuep|$ is bounded independently of $\ep$
 on the support of $\chi_R$, thus a subsequence of
 $\{|\nuep|\indic_{\supp\chi_R}\}_\ep$ converges to a positive
 measure $\tilde\nu$ of the form $2\pi \sum_{i=1}^k k_i \delta_{a_i}$,
 where $k_i$ is a positive integer for every $i$
 (the $a_i$'s are a   subset of $\Lambda$).

  From the weak convergence of $j_\ep$ to $j$ in $L^p_\loc$ and using  the inequality
  $|\nab_{A_\ep}u_\ep|\ge |j_\ep|$ (following from the assumption $|u_\ep|\le 1$)
  we have  for any $r>0$
\begin{equation}\label{ext}\liminf_{\ep\to 0}
\int_{\mr^2\sm\cup_{p\in\Lambda}
B(p,r)}  \chi_R |\nab_{A_\ep}
u_\ep|^2\ge\liminf_{\ep\to 0}\int_{\mr^2\sm\cup_{p\in\Lambda}
B(p,r)} \chi_R |j_\ep|^2 \ge\int_{\mr^2\sm\cup_{p\in\Lambda} B(p,r)}
\chi_R  |j|^2.\end{equation}
Indeed either the left-hand side is equal to $+\infty$ and the statement is true, or there is weak $L^2$ convergence of the currents on the complement of $\cup_p B(p,r)$ and \eqref{ext} follows by weak lower semicontinuity of the integrand.  Similarly, by weak
convergence of $h_\ep $ to $h$ we have
\begin{equation}\label{ext2}
\liminf_{\ep\to 0} \int_{\mr^2\sm\cup_{p\in\Lambda}
B(p,r)}  \chi_R
 {h_\ep}^2\ge\int_{\mr^2\sm\cup_{p\in\Lambda}
B(p,r)}
\chi_R h^2.\end{equation}

 Then consider  any $\eta\in(0,1)$   small enough
  so that the balls $B(a_i,2\eta)$ are disjoint. Note
 that since the limit of $|\nuep|$ on the support of $\chi_R$
 is a sum of Dirac masses concentrated at the points $\{a_i\}_i$   we have for $\ep$ small enough
$$|\nuep|(\supp\chi_R\sm\cup_i B(a_i,\eta)) = 0,\quad \nuep(B(a_i,\eta)) = 2\pi d_i,$$
where $2\pi d_i = \nu(a_i)$.

We use two distinct lower bounds for the integral of $\chi_R (e_\ep
- \dl\nuep)$ on balls. We distinguish the  set $I$ of indices
such that $B(a_i,2\eta)\subset \{\chi_R = 1\}$ and the remaining
indices  $J$. Note that if $i\in J$ then $B(a_i,2\eta)$ intersects
the set where $\chi_R\neq 1$ and the support of $\chi_R$, thus
$B(a_i,2\eta)\subset \URC\sm \URc$ for some $C>0$ independent of
$R>0$, $\eta\in (0,1)$ and $i$.

In the case $i\in I$ we use
\begin{equation}\label{uneboule}\int_{B(a_i,\eta)} e_\ep \ge \pi| d_i|\log \frac{\eta}\ep + C_{|d_i|} + o_{\eta,\ep}(1),\end{equation}
 where $C_{d}$ is a constant depending only on $d$ such that $C_{ 1} = \gamma$,
(where $\gamma$ is defined after Theorem~1), where  $C_0 = 0$, and where
$$\lim_{\eta\to 0}\limsup_{\ep\to 0}o_{\eta,\ep}(1)=0.$$
We postpone the proof of this well-known statement. It is very
similar to analogous ones found in \cite{bbh} or \cite{betriv}. Then
we deduce from \eqref{uneboule}  that
for any $i\in I$ and letting  $C_{d_i} = +\infty$ if $d_i < 0$,
\begin{equation}\label{unebonneboule}\liminf_{\ep\to 0}\int_{B(a_i,\eta)} \(e_\ep - \dl\,d\nuep\) \ge \pi d_i \log \eta\  + C_{d_i}  + o_\eta(1),\end{equation}
where $\lim_{\eta\to 0} o_\eta(1) = 0$.

In the case $i\in J$ we have to introduce the weight $\chi_R$ which
is no longer constant on the ball. Then we resort to Remark~\ref{rqdurbis}.
Consider the family of balls $\Cep$ consisting of the balls $B$ in
$\Betaep$ which intersect the support of $\chi_{R+1}$, and such that
$|\nuep|(B)\neq \varnothing$. For any $B\in\C_\ep$, since $|\nuep|(B)\neq 0$ and $|\nuep|\to 2\pi\sum_i k_i\delta_{a_i}$, and since $r(B)\le\eta/2$, we have for $\ep$ small enough depending on $R$ that there is some index $i$ for which $B\subset B(a_i,\eta)$.
Let $\Cepi$ denote the balls included in $B(a_i,\eta)$ and partition $\Cepi$ as $\cup_\alpha\Cepia$ and $\Cep$ as $\cup_\alpha \Cepa$ where the superscript $\alpha$ corresponds to the balls which are included in a ball   $B\in\Baep$ (we assume $\eta/2 <\rho$).

From \eqref{durbis}, for every $B\in\Cepia$
\begin{equation}\label{413}
\int_B \chi_R\(e_\ep - \Lambda_\ep^{\a,\eta/2}\,d \nuep\) \ge
- C |\nab\chi_R|_\infty |\nuep|(B) \ge -  C|\nuep|(B).\end{equation}
Now we note that since \eqref{hypcoerplus} holds, then  for $\ep$ small enough $\cba = M\naep$, for otherwise we would have $e_\ep(\Baep)\ge \frac{M}{3}\naep\lep$ and then that
$$\sum_{B\in\Baep} \gepB(B) =  \sum_{B\in\Baep} \fepB(B)  = \sum_{B \in \Baep} (e_\ep- \Lambda_\ep^\a  \nu_\ep) (B) \ge  (\frac{M}{3}-\pi)\naep\lep\quad\xrightarrow{\ep\to 0} \quad+\infty,$$
if we choose $M> 3 \pi$ and since $\naep\ge 1$. This is a contradiction with \eqref{hypcoer} since $g_\ep \ge \sum_B \gepB - C$ by \eqref{minmajgh}, proving that $\cba = M\naep$.

Then we have from \eqref{lamba} that
$$ \Lambda_\ep^{\a,\eta/2} - \dl = \hal\log\eta +\Delta,\quad\text{where}\quad |\Delta| \le C\(\log\naep+1\)$$
and
$$ \left|\int_B \(\chi_R - \chi_R(a_i)\)\,d\nuep\right| \le C\eta |\nuep|(B).$$
Hence with \eqref{413}
\begin{equation*}\begin{split}
\int_{B} \chi_R\(e_\ep -\dl\,d\nuep\) & = \int_B \chi_R\(e_\ep -
\Lambda_\ep^{\a,\eta/2}\,d\nuep\) +\(
\hal \log\eta + \Delta \)  \int_B \chi_R\,d\nuep\\
 &\ge - C  |\nuep|(B)
 + \frac{\log\eta}2 \chi_R(a_i)\nuep(B) -
 \frac\eta 2 \log\eta |\nuep|(B) - |\Delta ||\nu_\ep|(B)\\
 &\ge \frac{\log\eta}2 \nuep(B)\chi_R(a_i) - C|\nuep|(B) \(1+\log\naep\).
 \end{split}\end{equation*}
 Summing over $B\in\Cepia$ and then over $\a$ and  $i\in J$
 we find, since
 $$\sum_{B\in\Cepi} \nuep(B) = \nuep(B(a_i,\eta))\to \nu(B(a_i,\eta)) = 2\pi d_i,$$
 that,
 $$\liminf_{\ep\to 0} \int_{\cup_{i\in J} B(a_i,\eta)} \chi_R\(e_\ep - \dl\,d\nuep\)\ge \pi\sum_{i\in J} d_i\chi_R(a_i)\log \eta - C\Delta(R),$$
 where
 $$\Delta (R) = \limsup_{\ep\to 0}     \sum_{U_\a\subset \URC\sm\URc} \naep\(\log\naep +1\).$$

Summing \eqref{unebonneboule} over $i\in I$
 and adding  the above and \eqref{ext}--(\ref{ext2}), we deduce
\begin{multline}\label{presqueren}\liminf_{\ep\to 0}
\int\chi_R (e_\ep - \dl\,d\nuep)\ge \hal
\int_{\mr^2\sm\cup_{p\in\Lambda}
 B(p,\eta)}  \chi_R ( |j|^2 + h^2)  \\ +\sum_{i\in I }\chi_R(a_i)
 \(\pi
  d_i \log \eta + C_{d_i}\)
+ \sum_{i \in J} \chi_R (a_i)  \pi d_i \log \eta
   - C\Delta(R) - o_\eta(1).\end{multline}

We will now take the limit $\eta\to 0$ on the right-hand side.
For that we use a Hodge decomposition of $j$ in $B(a_i,\eta_0)$, writing $j = -\np H + \nabla K$, with $H = 0$ on $\p B(a_i,\eta_0)$.  Then since  $-\Delta H = \nu - h = 2\pi d_i\delta_{a_i} - 1$ we have  $H(x) =
d_i \log|x-a_i| + F$, where $F$ is in $H^2$  in the neighbourhood of
$a_i$, in particular $H\in W^{1,p}$ for any $p<2$, and since $j\in L^p$, this implies that $K\in W^{1,p}$ also. Then  an easy computation shows that
$$\lim_{\eta\to 0}\hal \int_{B(a_i,\eta_0)\sm B(a_i,\eta)}  \chi_R |\np H|^2 +\pi(\log \eta) {d_i}^2\chi_R(a_i)$$
exists and is finite, while
$$ \int_{B(a_i,\eta_0)\sm B(a_i,\eta)}  \chi_R |j|^2\ge \int_{B(a_i,\eta_0)\sm B(a_i,\eta)} \chi_R\(  |\np H|^2 + \np H\cdot\nab K\).$$
Decomposing $H$ and integrating by parts  we have, writing $C_{i,\eta} =  B(a_i,\eta_0)\sm B(a_i,\eta)$,
$$\int_{C_{i,\eta}}  \np H\cdot(\chi_R \nab K) = \int_{C_{i,\eta}} \np F\cdot(\chi_R\nab K) - d_i \int_{C_{i,\eta}} K\np \log  \cdot \nab \chi_R,$$
and this remains bounded as $\eta\to 0$, using the regularity of $\chi_R$, $F$, and the boundedness of $H$, $K$, $\log$ in $W^{1,p}$.
We may then  deduce that
$$\liminf_{\eta\to 0}\hal \int_{B(a_i,\eta_0)\sm B(a_i,\eta)}  \chi_R |j|^2 +\pi(\log \eta) {d_i}^2\chi_R(p)$$
is not equal to $-\infty$.

As a  consequence, writing $d_i = {d_i}^2 - ({d_i}^2 - d_i)$ in the right-hand side of \eqref{presqueren}, and this right-hand side being
bounded above  independently of $\eta$,  we have that $\sum_i ({d_i}^2 - d_i) \chi_R(a_i) \log\frac1\eta$ is bounded above as $\eta\to 0$. Thus  we have $d_i\in\{0,1\}$ for any $i$ such that $\chi_R(a_i)\neq 0$
 and then $d_i = 1$ since $d_i$ was assumed to be nonzero.
In view of this, (\ref{presqueren}) can be rewritten as
\begin{multline*} \liminf_{\ep\to 0} \int\chi_R (e_\ep -
\dl\,d\nuep)\ge \hal \int_{\mr^2\sm\cup_{p\in\Lambda}
 B(p,\eta)}  \chi_R (|j|^2 +h^2 )  \\ +\sum_{p \in \Lambda  }\chi_R(p)
 \(\pi
  \log \eta + \gamma\)
 - C\Delta(R) - o_\eta(1),\end{multline*}
 where we recall that $\gamma = C_1$ and we have absorbed $C_1\sum_{i\in J} \chi_R(a_i)$ in $C\Delta(R)$.

Letting $\eta\to 0$ we thus find (see \eqref{WR})
\begin{equation*}\liminf_{\ep\to 0}\int\chi_R
(e_\ep - \dl\,d\nuep)\ge W(j,\chi_R)+ \hal \int \chi_R h^2+
\sum_{p\in\Lambda}\chi_R(p)
 \gamma - C\Delta(R).\end{equation*}
 From \eqref{morebounds2} we may replace $e_\ep-\dl\nuep$ by $g_\ep$, with an error term which may be absorbed in $C\Delta(R)$ hence
\begin{equation}  \label{quasiren}\liminf_{\ep\to 0}
\int\chi_R \,d\gep \ge W(j,\chi_R)+\hal \int \chi_R h^2+
 \sum_{p\in\Lambda}\chi_R(p)  \gamma - C\Delta(R).\end{equation} Now, under hypothesis \eqref{hypcoerplus}
and using \eqref{calbound}, we have
$$\limsup_{\ep\to 0} \sum_{\alpha\mid A_\alpha\subset \UR} \naepp\le CR^2$$
and thus
$$\limsup_{R \to \infty}\limsup_{\ep \to 0}  \frac{1}{R^2}
\sum_{\alpha \mid A_\alpha \subset \URC\sm\URc} n_\ep^\alpha |\log
n_\ep^\alpha| =0.$$
Indeed, using H\"older's inequality,
and bounding the number of $\alpha$'s involved in the above sum by
$CR$, we find
$$\sum_{\alpha \mid A_\alpha \subset \URC\sm\URc)} [ n_\ep^\alpha]^{3/2}
\le (CR)^{1/4} \( \sum_{\alpha\mid A_\alpha\subset \URC}
\naepp\)^{3/4}\le C R^{1/4+3/2}.$$
It follows, since $U_\alpha
\subset A_\alpha$,  that
\begin{equation}\label{ren1} \limsup_{R\to+\infty}\frac{\Delta(R)}{R^2} = 0\end{equation}
and in particular $\nu(\URC\sm\URc) = o(R^2)$.
Then we write, using $\nu = \curl j + h$,
$$\sum_{p\in\Lambda} \chi_R(p) =
   \frac{1}{2\pi}
    \int \chi_R\,d\nu = \frac{1}{2\pi}
    \int \chi_R h - \frac{1}{2\pi} \int \np \chi_R\cdot j.$$
Let $E_R = \{0<\chi_R<1\}$. Then since $E_R\subset\URC\sm\URc$ we have $|E_R|\le CR$ and using \eqref{uniformes}  together with H�lder's inequality we find
$$ \int_{E_R} \chi_R h \le |E_R|^\hal \(\int_{E_R} h^2\)^\hal\le C R^{3/2},$$
and a similar bound for $\int \np \chi_R\cdot j$ using
\eqref{uniformes} again, since it is equal to $\int_{E_R} \np
\chi_R\cdot j$. Therefore
$$ \sum_{p\in\Lambda} \chi_R(p) = \frac{1}{2\pi}
\int_{\{\chi_R = 1\}} h + o(R^2) = \frac{1}{2\pi} \int_\UR h +
o(R^2),$$ the second equality being proved again with the help of
\eqref{uniformes} and H�lder's inequality. Together with
\eqref{ren1} and \eqref{quasiren}, this proves  \eqref{limlim}.

There remains to prove \eqref{uneboule}.  For this it is convenient to blow-up $B(a_i,\eta)$ to the unit ball $B_1$. Then \eqref{uneboule} becomes
\begin{equation}\label{bbhup}\hal\int_{B_1}\(|\nab_B v|^2 + \left|\frac{\curl B}{\eta}\right|^2 +\frac{(1-|v|^2)^2}{2{\ep'}^2}\)\ge
\pi |d_i|\log \frac1{\ep'} + C_{d_i} +
o_{\eta,\ep}(1),\end{equation} where $v(x) = u_\ep(\eta x)$,
$B(x) = \eta A_\ep(\eta x)$ and $\eta \ep' = \ep$, so   that  $\ep'$ tends to $0$ with
$\ep$. Note that $(v,B)$ depends on $\ep$ but we omit this in the
notation for the rest of the proof.

Since $\curl A_\ep\to h$ weakly  in $L^2_\loc$, it follows
 that $\|\curl B\|_{L^2(B_1)} \le 2 \eta  \|\curl A_\ep\|_{L^2(B_\eta)}\le C\eta$.  Then, choosing
 to work in the gauge $\div B = 0$, $B\cdot\tau = \text{constant}$ on
 $\partial B_1$, we have $\|B\|_{H^1(B_1)} \le C\eta$.  Since $j(u_\ep,A_\ep)$ is bounded in $L^p_\loc(\mr^2)$ for any $p<2$, we deduce immediately that $\|j(v,B)\|_{L^p(B_1)}\le C \eta^{1-2/p}$.
But by Sobolev embedding,  $\|B\|_{L^q(B_1)}=O(\eta)$ for any $q>1$ hence the integral of $B\cdot j(v,B)$ on $B_1$ is $o_{\eta}(1)$. Then, since
$$|\nab_B v|^2 = |\nab v|^2 - 2 B\cdot j(v,B) + |B|^2 |v|^2,$$
\eqref{bbhup} will follow if we  show that
\begin{equation}\label{bbhbup}\hal\int_{B_1}\(|\nab v|^2 + \frac{(1-|v|^2)^2}{2{\ep'}^2}\)\ge
\pi |d_i|\log \frac1{\ep'} + C_{d_i} +
o_{\eta,\ep}(1).\end{equation}

To prove \eqref{bbhbup}  we modify $B$ in order for the current to be divergence free: As before  we use the Hodge decomposition $j(v) := (iv,\nab v) = -\np H  +\nab K$ with $H = 0$ on $\p B_1$, and let $\tilde v  = ve^{-iK}$. Then denoting $e(v)$ the integrand in \eqref{bbhbup} we have
$$e(\tilde v) = e(v) - \nab K\cdot j(v) +\frac{ |v|^2}{2}  |\nab K|^2.$$
We replace  $j(v) = -\np H +\nab K$ and note that, integrating by
parts,  $\nab K\cdot\np H$ integrates to $0$ on $B_1$. Therefore
$$\int_{B_1} e(\tilde v) = \int_{B_1}\( e(v) +\(\frac{|v|^2}2 - 1\) |\nab K|^2\) \le
\int_{B_1} e(v) .$$
 Thus if we show the lower bound \eqref{bbhbup}
for $\tilde v$, then we are done. For this we may  assume, without loss of generality, that the upper bound
\begin{equation}\label{bbhsup}\hal\int_{B_1}\(|\nab \tilde v|^2 + \frac{(1-|\tilde v|^2)^2}{2{\ep'}^2}\)\le
\pi |d_i|\log \frac1{\ep'} + C_{d_i}\end{equation}
holds.

The advantage is that now we have
$$j(\tilde v) = - \np H + (1 - |v|^2) \nab K.$$
But  $\lim_{\ep\to 0}(1 - |v|^2) =  0$ in $L^q(B_1)$ for any $q>1$,
 being bounded in $L^\infty$ and tending to $0$ in $L^2$.
 Moreover, we have seen that
    $\|j(v,B)\|_{L^p(B_1)}\le C \eta^{1-2/p}$, and  that
 $B = O(\eta)$ in every $L^p$, so
\begin{equation}\label{courantsproches} j(v,B) - j(v) = |v|^2 B = O(\eta)\end{equation}
and therefore $j(v) = O(\eta^{1-2/p})$ in $L^p$, which implies that
$H$ and $K$ are  $O(\eta^{1-2/p})$ in $W^{1,p}$. It follows from the
above that
\begin{equation}\label{tranche1} j(\tilde v) + \np H = o_{\eta,\ep}(1).\end{equation}
in $L^p(B_1)$, for every $p<2$.

Moreover,  since $\curl j(u_\ep,A_\ep) + h_\ep \to 2\pi
d_i\delta_{a_i}$ in $W^{-1,p}$ as $\ep\to 0$,  we have that $\curl
j(v,B) + \eta \, \curl B \to 2\pi d_i\delta_0$. Hence using
\eqref{courantsproches} we deduce $-\Delta H = \curl j(v) \to 2\pi
d_i\delta_0  + o_\eta(1)$ as $\ep\to 0$ in $W^{-1,p}$. Since $H = 0$
on $\p B_1$ we then have
\begin{equation}\label{tranche2} H(x) = -2\pi d_i\log|x| + o_\eta(1)\end{equation}
in $W^{1,p}$.

From \eqref{tranche1}, \eqref{tranche2} we may find radii $\{r_\ep\}_\ep$ such that
$$\text{i) $\lim_{\ep\to 0} r_\ep = 1$},\quad\text{ii)  $\|j(\tilde v) + \np H\|_{L^p(\p B_{r_\ep})} = o_{\eta,\ep}(1)$},\quad \text{iii) $\|H + 2\pi d_i\log\|_{W^{1,p}(\p B_{r_\ep})} = o_{\eta}(1)$}.$$
We may further require that  $\rho:=|\tilde v|\to 1$ uniformly as $\ep\to 0$ on $\partial
B_{r_\ep}$. Indeed  from \eqref{bbhsup}  we have
$$\hal \int_{B_1} |\nab \rho|^2 +  \frac1{2{\ep'}^2} (1-\rho^2)^2
\le C\log\ep'$$
thus  a mean value argument easily implies that $r_\ep$ may be chosen such that
$$\hal \int_{\p B_{r_\ep}} |\nab \rho|^2 +  \frac1{2{\ep'}^2} (1-\rho^2)^2\le C (\log\ep')^2.$$
This in turn  implies using \eqref{esmodcir} that $\|\rho-1\|_{L^\infty(\p B_{r_\ep})}\to 0$ as $\ep\to 0$.

Then, writing $\tilde v = \rho e^{i\vp}$, we have $j(\tilde v) = \rho^2 \nab\vp$ and the above implies
 that  for some $\theta_0\in\mr$,
$$ \tilde v = (1+\tilde{\rho}) e^{i(\theta_0+ d_i\theta + \tilde{\vp})},
\quad\text{where $ \|\tilde{\vp}\|_{W^{1,p}(\partial B_1)} =
o_{\eta,\ep}(1)$
 and $\|\tilde{\rho}\|_{L^\infty(\partial B_1)} = o_\ep(1)$.}$$
Without going into further detail (see for instance \cite{bbh}, Chapter {\textsc{\romannumeral 8}}),
the above implies that
\begin{multline*}
 \hal\int_{B_1}\(|\nab \tilde v|^2 +
\frac{(1-|\tilde v|^2)^2}{2{\ep'}^2}\)\\
 \ge \min\left\{\hal\int_{B_1}\(
|\nab u|^2 +\frac1{2{\ep'}^2} (1-|u|^2)^2\) \mid \text{$u =
e^{id_i\theta}$ on $\partial B_1$}\right\} +
o_{\eta,\ep}(1).\end{multline*}
From \cite{bbh}, the right-hand side
is precisely equal to $\pi |d_i|\log \frac1{\ep'} + C_{|d_i|}
+o_\ep(1)$, where the constant $C_d$ is equal to $\gamma$ if $d =
 1$. Thus we have proved \eqref{bbhbup}, and then
\eqref{uneboule}.

\section{Proof of Proposition \ref{jerr}}

The proof of Proposition~\ref{jerr}  is based on the ball construction of R. Jerrard
\cite{jerr}, hence  we will only  emphasize
the points which need some modification, mostly to take into account
the presence of the magnetic potential $A$ the way we do in
\cite{livre}. We will denote by $c$, $C$, respectively, a small and a large generic universal constant. We will number the constants we need to keep track of. Throughout this section $U$ is a bounded domain in $\mr^2$ and $(u,A)$ are defined on $U$.

The first ingredient is  a  lower bound for the energy of $|u|$  on
a circle (\cite{jerr} Lemma 2.3). It is valid for any $\ep>0$.
\begin{lem}
Assuming  $2 r\ge \ep>0$ and $x$ are such that the closed ball $B(x,r)\subset U$ we have
\begin{equation} \label{esmodcir} \hal \int_{\p B(x, r)}
|\nab |u||^2 + \frac{(1-|u|^2)^2}{2\ep^2}\ge c_0
\frac{(1-m)^2}{\ep},\end{equation} where $m= \min_{\p B(x,r)} |u|$.
\end{lem}

 In contrast to \cite{jerr} and because we wish to work with constants independent of $U$ we introduce
$$U_\ep = \{x\in U\mid \dist(x,U^c) >\ep\}.$$
Then  $u:U\to \C$ being given
 we introduce, following \cite{jerr},  $S = \{x\in U_\ep\mid |u|\le 1/2\}$. Assuming $u$ is continuous the connected  components of $S$ which are included in $U_\ep$ are compact, and $u/|u|$ has a well defined degree, or winding number on their boundary. Then we let
$$S_E = \{\text{Union of the components of $S$ with nonzero boundary degree}\}.$$
Still following \cite{jerr}, for any compact $K\subset U$ such that $\p K\cap S_E = \varnothing$ we let
$$ \deg_E(u,\p K) = \sum_{\text{$S_i$ component of $S_E$}} \deg(u,\p S_i).$$
Note that this degree is defined even if $|u|$ vanishes on $\p K$, provided the points where it vanishes are not in $S_E$.

The previous lemma implies (see \cite{jerr}, Proposition~3.3)
\begin{lem}\label{lem01}
There exists   a collection of disjoint closed  balls $B_1, \cdots, B_k$ of radii $r_1, \cdots, r_k$
 such that
\begin{enumerate}
\item $\forall i, r_i \ge \ep$
\item $ S_E\cap U_\ep \subset \cup_{i=1}^k B_i$
\item $\forall i, e_\ep(U\cap B_i)\ge c_1 r_i/\ep$.
\end{enumerate}
 \end{lem}

\begin{proof} We only sketch the proof. If
$x\in S_E$ then either $\p B_r(x)$ intersects
$\{|u| \le 1/2\}$ for every $\ep/2\le r\le\ep$ and  the above
lemma implies that $e_\ep(U\cap B(x,\ep))\ge c$ or there exists
 $\ep/2\le r\le \ep$ such that $|u|> 1/2$ on $\p B_r(x)$ and
 then the connected component  of $x$ in $S_E$, which has nonzero
 degree, is  included in $B(x,r)$. The nonzero degree  implies   again
 (see \cite{jerr})
 $e_\ep(U\cap B(x,\ep))\ge c$. We thus have a cover of $S_E$ by balls which satisfy $e_\ep(B)\ge c r(B)/\ep$.

From Besicovitch's Lemma, there exists a disjoint subcollection
$\{B_k\}_k$ such that $\{\widetilde B_k\}_k$ covers $S_E$, where
$\widetilde B_k = C B_k$ and $C$ is a universal constant. These
balls still satisfy  $e_\ep(B)\ge c r(B)/\ep$, though with a smaller
constant.  Then, grouping the balls which intersect in larger ones
as in \cite{jerr} (see also  \cite{ss1}) allows to obtain a disjoint
cover of $S_E$ with the same property. Item 1) is trivially verified
since the balls we started with had radius $\ep$. Note also that the
balls obtained here only depend on $S_E$ hence on $u$.

\end{proof}

Still following  \cite{jerr}, we have:
\begin{pro}Choose $c_2\in (0,c_1) $ small enough and let
$$\lambda_\ep(x) = \min\(\frac{c_2}\ep, \frac\pi x\frac1{1+\frac x2+\frac{\pi\ep}{c_0 x}}\).$$

Then, assuming that  $B_r\subset U_\ep$, that $\p B_r\cap S_E =
 \varnothing$  and that $\ep\le r\le |d|/2$, where  $d = \deg_E(u,\p B_r)$ is assumed to be different from $0$, we have
\begin{equation}
\label{estncir} \hal \int_{\p B_r}  |\nab_{A} u |^2 +
\frac{1}{2}\int_{B_r}|\curl A  |^2 +\frac1{4\ep^2}\int_{ \p B_r}\(1 - |u|^2\)^2\ge \lambda_\ep\(
\frac{r}{|d|}\) .\end{equation}
Moreover, the primitive function $\Lambda_\ep(x) =
 \int_0^x \lambda_\ep$ is increasing, $s\mapsto \Lambda_\ep(s)/s$ is decreasing,
$$\lim_{s\searrow 0}  \frac{\Lambda_\ep(s)}s  = \frac{\min(c_0, c_2)}\ep<\frac{c_1}{\ep},\quad \frac{\Lambda_\ep(\ep)}\ep \ge \frac{c_3}\ep$$
and finally, for any $\ep\le s\le 1/2$, and for some $C_0>0$,
\begin{equation}\label{propL}
\Lambda_\ep(s) \ge \pi \log \frac{s}{\ep} - C_0.
\end{equation}
\end{pro}
\begin{proof} First, in   the case where $\p B_r$ intersects $\{|u|\le 1/2\}$ we deduce from  \eqref{esmodcir}  that \eqref{estncir} is satisfied with $c_2 = c_0/4$.

When on the contrary $|u|>1/2$ on $\p B_r$ we have $\deg_E(u,\p B_r)
= \deg(u,\p B_r)$. Then we bound from below $\hal\int_{\p B_r}
|u|^2 |\nab \vp-A|^2 $, where $u = |u|e^{i\vp}$ as
follows: Still denoting  $m= \min_{\p B_r} |u|$, using the
Cauchy-Schwarz inequality  we have
$$ \hal\int_{\p B_r} |u|^2 |\nab \vp-A|^2 \ge \frac{m^2}{2}\frac{1}{2\pi r} \( \int_{\p B_r} \frac{\p \vp}{\p \tau} - A \cdot \tau\)^2  = \frac{m^2}{4\pi r} ( 2\pi d - X  )^2  $$
where we write $X:=  \int_{B_r} \curl A  = \int_{\p B_r} A
\cdot \tau   $. On the other hand, by Cauchy-Schwarz
again
 $$\frac{1}{2} \int_{B_r} |\curl A |^2 \ge \frac{1}{2\pi r^2 }\( \int_{B_r} \curl A  \)^2 =
 \frac{ X^2}{2\pi r^2 }
 $$
 Adding the two relations we obtain
$$ \hal\int_{\p B_r} |u|^2 |\nab \vp-A|^2+ \frac{1}{2} \int_{B_r} |\curl A  |^2\ge \frac{1}{2\pi r} \( \frac{m^2}{2} (2\pi d - X  )^2 +  \frac{1}{r} X^2\).
$$
Minimizing the right-hand side with respect to $X$ yields
\begin{equation}\label{lbcircle} \hal\int_{\p B_r} |u|^2 |\nab \vp-A|^2+ \frac{1}{2} \int_{B_r} |\curl A  |^2\ge \frac{\pi d^2}{r}\frac{m^2}{1+\frac{m^2 r}2}.\end{equation}
Adding \eqref{esmodcir} we deduce for $r\ge\ep$ that
\begin{equation}\label{lbcircle2} e_\ep(\p B_r) \ge \frac{\pi |d|}r\frac{|d|}{\frac1{m^2} + \frac r2} + c_0\frac{(1-m)^2}{\ep}.\end{equation}

If $|d|>1$, then either $m^2<2/3$ and we find $e_\ep>c/\ep$ for a
well chosen $c>0$  or $m^2\ge 2/3$ and, since $r/2 < |d|/4$, we have
$m^{-2} + r/2\le 3/2 +  |d|/4 \le |d|$ implying $e_\ep\ge \pi
|d|/r$. Thus, if $|d|>1$, \eqref{estncir} is satisfied. If $|d|=1$
then minimizing the right-hand side of \eqref{lbcircle2} with
respect to $m$ yields
$$e_\ep(\p B_r)\ge \frac\pi r\frac1{1+\frac r2+\frac{\pi\ep}{c_0 r}},$$
so that in every case we have $e_\ep(\p B_r)\ge \lambda_\ep(r/|d|),$ if $c_2$ is chosen small enough.

We now turn to the properties of $\Lambda_\ep$. Since $\lambda_\ep$
is positive, decreasing, then $\Lambda_\ep$ is increasing and
$\Lambda_\ep(s)/s$ is decreasing. It is clear that as $s \to 0$, we
have  $\lambda_\ep(s) \sim  \min(c_0,c_2)/\ep\sim \Lambda_\ep(s)/s$.
Moreover, if $x> c\ep$, with $c = \pi/c_2$, then
$$\lambda_\ep(x) = \frac\pi x\frac1{1+\frac x2+\frac{\pi\ep}{c_0 x}}$$
hence, if $s\ge c \ep$,
\begin{equation*}\begin{split}\Lambda_\ep(s)&\ge\int_{c\ep}^s
\frac\pi x\frac1{1+\frac x2+\frac{\pi\ep}{c_0 x}}\,dx\ge
 \int_{c\ep}^s \frac\pi x\(1 - \frac x2 - \frac{\pi\ep}{c_0x}\) dx
 \\ &\ge \pi\log\frac s\ep - C_0,\end{split}\end{equation*}
for some   constant $C_0$. If $s < c \ep$ then the inequality remains true if $C_0$ is chosen large enough, since $\Lambda_\ep(s)\ge 0.$

Finally, $\Lambda_\ep(\ep)\ge \ep \lambda_\ep(\ep) \ge c_3$, if
$c_3>0$ is chosen small enough.
\end{proof}

 From there, the ball construction
procedure (growing and merging of balls) from \cite{jerr} (or see
\cite{ss1} Prop 3.1) allows to deduce
\begin{pro} \label{pros} For any $0<s<1/2$ there exists a family of disjoint closed balls $\B(s)$ (depending only on $u_\ep$)
  such that
\begin{enumerate}
\item The family of balls is monotonic i.e. if $s<t $, we have $\B(s)\subset \B(t)$. Moreover, denoting by  $r(B)$ the radius of $B$, the function $s\to \sum_{B\in\B(s)} r(B)$ is continuous.
\item For any $s$ we have $S_E\subset \B(s)$.
\item For any $B\in\B(s)$
$$ e_\ep(U\cap B) \ge r(B) \frac{\Lambda_\ep(s)}{s}.$$
\item If $B\in\B(s)$ and $B\subset U_\ep$ then, letting $d_B = \deg_E(u_\ep,\p B)$, we have $  r(B) \ge  s| d_B|$.
\end{enumerate}
\end{pro}
\begin{proof} We let  $\B(s_0)$  be the family of balls given
 by  Lemma \ref{lem01}, where we choose $s_0$ small enough so that items
 3 and 4 are satisfied (item 2 obviously is).
  We let $\B(s) = \B(s_0)$ for every $s\le s_0$.
    For $s\ge s_0$ we apply the method of growing and merging
    of \cite{jerr} which we sketch briefly:
    It consists in continuously increasing the parameter $s$ and
    at the same time making those balls included in $U_\ep$
    such that $r(B) = s |d_B|$ grow so that the equality remains
    satisfied. When balls touch, the parameter  $s$ is stopped and
    the balls are merged into a larger ball with radius the sum
    of the radii of the merged balls, and  this is repeated if
    the resulting family is still not disjoint. This  does not
    change the total radius and when it is done, i.e. when the
    family is disjoint again, the increasing of $s$ is resumed,
     etc...
      This yields a family of disjoint closed balls which is monotonic, such that $s\to \sum_{B\in\B(s)} r(B)$ is continuous and such that $r(B) \ge  s| d_B|$ for every ball included in $U_\ep$. Obviously $S_E\cap U_\ep\subset \B(s)$ for every $s$.
Also the growing and merging process depends only on the initial
balls and the degrees of $u_\ep$, hence on $u_\ep$.

 The lower bound  $ e_\ep(U\cap B) \ge r(B) \Lambda_\ep(s)/s$ is true initially and is preserved through the merging process, it is also preserved through the growing process  as long as \eqref{estncir} remains valid, i.e. $r(B)<|d_B|/2$ for every $B\subset U_\ep$ such that $d_B\neq 0$.  This results from the properties of $\Lambda_\ep$, as detailed in \cite{jerr}. Then for the process to stop, there must be a ball $B$ for which $r(B) = s |d_B|$, i.e. a growing ball, with  $r(B)\ge |d_B|/2$, hence we must have $s\ge 1/2$.
 \end{proof}

We may now deduce
\begin{proof}[Proof of Proposition \ref{jerr}] We begin
by constructing a family $\B'(s)$ which contains $S_E$
 instead of $\{x\in U_\ep\mid |u|\le 1/2\}$ but satisfies items 2--3 and  then modify it.

Under the hypotheses, Proposition~\ref{pros} applies, and yields for every $0<s<1/2$ a family of balls $\B'(s)$ satisfying the 4 items stated.
Choosing $s_0$ small enough we have $\Lambda(s_0)/s_0\ge c /\ep$ hence, letting $r_0$ denote the total radius of the balls in $\B'(s_0)$,
$$\ep^{-\beta}\ge G_\ep(u, A) \ge \frac{c  r_0}{\ep}$$
and therefore $r_0\le C \ep^{1-\beta}$.

Let $r\in(C\ep^{1-\beta}, 1/2)$, and let $r_1$ denote the total
radius of the balls in $\B'(1/2)$. If $r>r_1$ then $\B'(1/2)$
 satisfies item  2 trivially and moreover for any $B\in \B'(1/2)$ we have from Proposition~\ref{pros} and using \eqref{propL} that
  $$e_\ep(B)\ge |d_B|\Lambda_\ep\(\hal\)\ge \pi|d_B|\(\log\frac1{2\ep} -C\)\ge \pi|d_B|\(\log\frac r{\cba\ep} -C'\),$$
  for any $r\le 1/2$ and any $\cba\ge 2$, proving item~3 in this case.

  If $r<r_1$ then there exists $s\in(s_0,1/2)$ such that $\B':=\B'(s)$ satisfies $r(\B') = r$. Then item 2 of the proposition is satisfied for this collection.
  Let us check item~3.

Assume then $e_\ep(\B') \le \overline{C} \log(r/\ep), $ with $2\le\overline{C}\le (r/\ep)^\hal $. We show by contradiction that if $M$ is chosen large enough, then
$$s\ge\frac r{M\overline{C}}.$$

Since $e_\ep(\B')\ge r\Lambda_\ep(s)/s$ and since
$\Lambda_\ep(s)/s$ is decreasing, if $s< r/(M\overline{C})$ and
 $ r/(M\overline{C})\le \hal$ then
$$\overline{C}\log\frac r\ep \ge M\overline{C} \Lambda_\ep\(\frac r{M\overline{C}}\)\ge \pi M\overline{C} \log\(\frac r{\ep M\overline{C}}\) - C_0  M\overline{C}.$$
It follows that
$$(1-\pi M) \log\frac r\ep + \pi M\log\overline{C} +\pi M\log M - C_0 M\ge 0,$$
which yields a contradiction for $M = 3/\pi$ and $r\ge C \ep$, with $C $ large enough, recalling that $\overline{C}\le (r/\ep)^\hal $. Therefore $s\ge \pi r/(3\overline{C})$ and then for every $B\in\B'$ such that
$B\subset U_\ep$ we have
$$e_\ep(B) \ge  r(B) \frac{ \Lambda_\ep(s)}{s}\ge |d_B| \Lambda_\ep(s) \ge |d_B| \Lambda_\ep\(\frac{\pi r}{3\overline{C}}\),$$
which in view of \eqref{propL} yields $\forall B \in \B'$ such that
$B\subset U_\ep$
$$e_\ep(B) \ge \pi  |d_B|\( \log \frac{r}{\ep \overline{C}}  -C \),$$
if $C$ is chosen large enough.

It remains to modify $\B'(s)$ so that $S:=\{x\in U_\ep\mid |u|\le 1/2\}\subset \B(r)$. First we note that a well known application of the coarea formula yields rather easily (see \cite{livre}, Proposition~4.8) that $S$ can be covered by a collection of disjoint closed balls $\C$ such that $r(\C)\le C \ep G_\ep \le C\ep^{1-\beta}$.  Then   for every $s$ we do the merging of the balls in $\C\cup \B'(s)$ as in the proof of Proposition~\ref{pros} to obtain $\B(s)$. If we chose $s$ such that  $r(\B'(s)) = r/2$  with $C\ep^{1-\beta} < r < 1$ and $C$ large enough, then $r(\B(s))\le r$ since $r(\C)\le   C\ep^{1-\beta}$. Moreover, if $B\in\B(s)$ is such that $B\subset U_\ep$ then $\deg(u,\p B)$ is the sum of $\deg_E(u, \p B')$ for $B'\in\B'(s)$ and $B'\subset B$. Then, if $e_\ep(B)\le \overline{C} \log (r/2\ep)$ the same bound holds for the $B'$'s and summing the above lower bounds we find
$$e_\ep(B) \ge \pi  |d_B|\( \log \frac{r}{2\ep \overline{C}}  -C \).$$
Changing the constant $C$ we can get rid of the factor 2 and $\B(s)$ has all the desired properties.
\end{proof}

\noindent
{\sc Etienne Sandier}\\
Universit\'e Paris-Est,\\
LAMA -- CNRS UMR 8050,\\
61, Avenue du G\'en\'eral de Gaulle, 94010 Cr\'eteil, France\\
\& Institut Universitaire de France\\
{\tt sandier@u-pec.fr}\\ \\
{\sc Sylvia Serfaty}\\
UPMC Univ  Paris 06, UMR 7598 Laboratoire Jacques-Louis Lions,\\
 Paris, F-75005 France ;\\
 CNRS, UMR 7598 LJLL, Paris, F-75005 France \\
 \&  Courant Institute, New York University\\
251 Mercer st, NY NY 10012, USA\\
{\tt serfaty@ann.jussieu.fr}

\end{document}